\documentclass[article,12pt]{amsart}

\usepackage[T1]{fontenc}
\usepackage{inputenc}
\usepackage[francais,english]{babel}
\usepackage{amsmath}
\usepackage{latexsym}
\usepackage{marvosym}
\usepackage{amsfonts}
\usepackage{amsthm}
\usepackage{float}
\usepackage{color}
\usepackage{epsfig}
\usepackage{graphicx}
\usepackage{amssymb}

 \newcommand{\E}{\mathbb{E}}

\newcommand{\R}{\mathbb{R}}

\newcommand{\dR}{\mathbb{R}}
\newcommand{\dE}{\mathbb{E}}

\newcommand{\cF}{\mathcal{F}}

\newcommand{\rI}{\mathrm{I}}
\newcommand{\wh}{\widehat}

\newcommand{\ind}{\mbox{1}\kern-.25em \mbox{I}}
\font\calcal=cmsy10 scaled\magstep1
\def\build#1_#2^#3{\mathrel{\mathop{\kern 0pt#1}\limits_{#2}^{#3}}}
\def\liml{\build{\longrightarrow}_{}^{{\mbox{\calcal L}}}}

\def\videbox{\mathbin{\vbox{\hrule\hbox{\vrule height1ex \kern.5em
\vrule height1ex}\hrule}}}

\numberwithin{equation}{section}
\theoremstyle{plain}
\newtheorem{thm}{Theorem}[section]
\newtheorem{rem}{Remark}[section]
\newtheorem{lem}{Lemma}[section]


\begin{document}
\title[Parametric estimation of deformations on random variables]
{A Robbins Monro procedure for the estimation of parametric deformations on random variables}
\author{Philippe Fraysse}
%\dedicatory{\normalsize Universit\'e de Bordeaux}
\address{Universit\'e de Bordeaux, Institut de Math\'ematiques de Bordeaux, UMR CNRS 5251 
F-33400 Talence, France.}
\author{Helene Lescornel}
%\dedicatory{\normalsize Universit\'e Paul Sabatier}
\address{Universit\'e Paul Sabatier, Institut de Math\'ematiques de Toulouse, UMR CNRS 5219 
F-31062 Toulouse, France.}
\author{Jean-Michel Loubes}
%\dedicatory{\normalsize Universit\'e Paul Sabatier}
%\address{Universit\'e Paul Sabatier, Institut de Math\'ematiques de Toulouse, UMR CNRS 5251 
%F-31062 Toulouse, France.}

\begin{abstract}
 The paper is devoted to the study of a parametric deformation model of independent and identically random variables. Firstly, we construct an efficient and very easy to compute recursive estimate of the parameter. Our stochastic estimator is similar to the Robbins-Monro procedure where the contrast function is the Wasserstein distance. Secondly, we propose a recursive estimator similar to that of Parzen-Rosenblatt kernel density estimator in order to estimate the density of the random variables. This estimate takes into account the previous estimation of the parameter of the model. Finally, we illustrate the performance of our estimation procedure on simulations for the Box-Cox transformation and the arcsinh transformation.
\end{abstract}

\maketitle

%%%%%%%%%%%%%%%%%%%%%%%%%%%%%%%%%%%%%%%%%%%%%%%%%%%%%%%%%%%%%%%%%%%%%%%%%%%%%%%%

\section{INTRODUCTION}
In many situations, random variables are not directly observed but only their image by a deformation is available. Hence, finding the {\it mean behaviour} of a  data sample becomes a difficult task since   the usual notion of Euclidean mean is too rough when the information conveyed by the data possesses an inner geometry far from the Euclidean one. Indeed, deformations on the data such as translations, scale location models for instance or more general warping procedures prevent the use of the usual methods in data analysis.  \vskip .1in
On the one hand, the deformations may result from some variations which are not directly correlated to the studied phenomenon. This situation occurs often in biology  for example when considering gene expression data obtained from microarray technologies to measure genome wide expression levels of genes in a given organism as described in \cite{Bolstad-03}.  A natural way to handle this phenomena is to remove these variations in order to align the measured densities. However, it is quite difficult to implement since the densities are unknown. In bioinformatics and computational biology, a method to reduce this kind of variability is known as normalization (see \cite{ GALLON-2011-593476} and references therein). \\ In epidemiology, removing variations is important in medical studies, where one observes age-at-death of several cohorts. Indeed, the individuals or animals members of the cohort enjoy different life conditions which means that time-variation  is likely to exist between the cohort densities and hazard rates due the effects of the different biotopes  on aging. Synchronization of the different observations is thus a crucial point before any statistical study of the data. \vskip .1in
On the other hand,  the variations on the observations are often due to transformations that have been conducted by the statisticians themselves. In econometric science, transformations have been used to aid interpretability as well as to improve statistical performance of some indicators. An important contribution to this methodology was made by Box and Cox in \cite{MR0192611}  who proposed a parametric
power family of transformations that nested the logarithm and the level. Estimation in this framework is achieved in \cite{MR2396812}.\vskip .1in
In this work, we concentrate on the case where the data and their transformation are observed in a  sequence model defined, for all $n\geq0$, by
\begin{equation}
\label{model}
X_{n}=\varphi_{\theta}(\varepsilon_{n})
\end{equation}
where, for all $t\in{\mathbb{R}}$, the family of parametric functions $\left(\varphi_{t}\right)$ is known and $\left(\varepsilon_{n}\right)$ is a sequence of independent and identically distributed random variables.
Our main goal is to estimate recursively the unknown parameter $\theta$ by privilegiating an alignment in distribution. More precisely, our approach to estimate $\theta$ is associated with a stochastic recursive 
algorithm similar to that of Robbins-Monro described in \cite{RobbinsMonro51} and \cite{RobbinsSiegmund71}.
\par
Assume that one can find a function $\phi$ (called contrast function) free of the parameter $\theta$, such that $\phi(\theta)=0$. Then, it is possible to estimate
$\theta$ by the Robbins-Monro algorithm
\begin{equation}
\label{RMalgo}
\widehat{\theta}_{n+1}=\widehat{\theta}_{n}+\gamma_{n}T_{n+1}
\end{equation}
where $(\gamma_n)$ is a positive sequence of real numbers decreasing towards zero and
$(T_n)$ is a sequence of random variables such that 
$\dE[T_{n+1}|\cF_n]=\phi(\widehat{\theta}_{n})$
where $\mathcal{F}_{n}$ stands for the $\sigma $-algebra of the events occurring up to time $n$.
Under standard conditions on the function $\phi$ and on the sequence $(\gamma_n)$, it is well-known (see in \cite{Duflo97} and \cite{KushnerYin03}) that
$\widehat{\theta}_{n}$ tends to $\theta$ almost surely. The asymptotic normality of $\widehat{\theta}_{n}$ together
with the quadratic strong law may also be found in \cite{HallHeyde80}. A randomly truncated
version of the Robbins-Monro algorithm is also given in  \cite{Chen88}, \cite{Lelong08}, whereas we can find in \cite{BF10} an application of the Robbins-Monro algorithm in semiparametric regression models.
In our framework, if we assume that $\varphi_{t}$ is inversible, then one can consider
$$
Z_{n}(t)=\varphi_{t}^{-1}\left(X_{n}\right).
$$
Hence, a natural registration criterion is to minimize with respect to $t$ the quadratic distance between $Z_{n}(t)$ and $\varepsilon_{n}$
$$
M(t)=\E \left[ \left\vert Z_{n}(t) - \varepsilon_{n}\right\vert^ 2 \right].
$$
%where $F$ is the common distribution function of $\left(\varepsilon_{n}\right)$ and $F_{Z(t)}$ is the one of $Z_{n}(t)$. 
It is then obvious that the parameter $\theta$ is a global minimum of $M$ and one can implement a Robbins-Monro procedure for the contrast function $M^{\prime}$, which is the differential of the $L^2$ function $M$.
\par
The second part of the paper concerns the estimation of the density $f$ of the random variables $\left(\varepsilon_{n}\right)$. More precisely, we focus our attention on the Parzen-Rosenblatt estimator of $f$ described for instance in \cite{Parzen62} or \cite{Rosenblatt56} .
Under reasonable conditions on the function $f$, Parzen  established in \cite{Parzen62} the pointwise convergence in probability and the asymptotic normality of the estimator without the parameter $\theta$. In \cite{Silverman78},  Silverman obtained uniform consistency properties of the estimator. Moreover, important contributions on the $L^{1}$-integrated risk has been obtained by Devroye in \cite{Devroye88} whereas Hall  has studied in \cite{Hall82} and \cite{Hall84} the $L^{2}$-integrated risk. In our situation, we propose to make use of a recursive Parzen-Rosenblatt estimator of $f$ which
takes into account the previous estimation of the parameter $\theta$.
It is given, for all $x\in \dR$, by
\begin{equation}
\label{RNW}
\widehat{f}_{n}(x)=\frac{1}{n}\sum_{i=1}^{n} W_{i}(x)
\end{equation}
with
$$
W_{i}(x)=\frac{1}{h_{i}}K\left(\frac{x-Z_{i}(\widehat{\theta}_{i-1})}{h_{i}}\right)
$$
where the kernel $K$ is a chosen probability density function and  
the bandwidth $(h_i)$ is a sequence of positive real numbers  
decreasing to zero. The main difficulty arising here is that we have to deal with the term $Z_{i}(\widehat{\theta}_{i-1})$ inside the kernel $K$. 
 \vskip .1in
The paper falls into the following parts. Section 2 is devoted to the description of the model. Section 3 
deals with the parametric estimation of $\theta$. We establish the almost sure convergence of $\widehat{\theta}_{n}$ 
as well as its asymptotic normality. In Section 4, under standard regularity assumptions on the kernel $K$,
we prove the almost sure pointwise and quadratic convergences of $\widehat{f}_{n}(x)$ to $f(x)$. Section 5 contains some numerical experiments on the well known Box-Cox transformation and on the arsinh transformation illustrating
the performances of our parametric estimation procedure. 
The proofs of the parametric results are given is Section 6, while those concerning 
the nonparametric results are postponed to Section 7.

%This problem arises naturally for a wide range of statistical research fields such as functional data analysis for instance in \cite{Gamboa-Loubes-Maza-07}, \cite{dupuy2011non}, \cite{Ramsay-Silverman-05}, \cite{bercufraysseaos}  and references therein, image analysis in \cite{TrouveY05}  or \cite{JASA}, shape analysis in \cite{kendall} or \cite{bb206} with many applications ranging from biology in \cite{Bolstad-03} to pattern recognition \cite{Sakoe-Chiba-78} just to name a few
%%%%%%%%%%%%%%%%%%%%%%%%%%%%%%%%%%%%%%%%%%%%%%%%%%%%%%%%%%%%%%%%%%%%%%%%%%%%%%%%

%%%%%%%%%%%%%%%%%%%%%%%%%%%%%%%%%%%%%%%%%%%%%%%%%%%%%%%%%%%%%%%%%%%%%%%%%%%%%%%%

\section{DESCRIPTION OF THE MODEL AND THE CRITERION}

%%%%%%%%%%%%%%%%%%%%%%%%%%%%%%%%%%%%%%%%%%%%%%%%%%%%%%%%%%%%%%%%%%%%%%%%%%%%%%%%

Suppose that we observe independent and identically distributed random variables $\varepsilon_{n}$ and a deformation $X_{n}$ of $\varepsilon_{n}$ according to the model \eqref{model} defined, for all $n\geq0$, by 
$$
X_{n}=\varphi_{\theta}(\varepsilon_{n})
$$
where $\theta\in \Theta \subset \dR$. Throughout the paper, we denote by $\varepsilon$ and $X$ random variables sharing the same distribution as $\varepsilon_{n}$ and $X_{n}$, respectively.

Assume that for all $t\in{\dR}$, the family of parametric functions $\left(\varphi_{t}\right)$ is known but  that the parameter $\theta$ is unknown. This situation corresponds to the case where the warping operator can be modeled by a  parametric shape. Estimating the parameter is the key to understand the amount of deformation in the chosen deformation class. This model has been widely used in the regression case, see for instance in \cite{ Gamboa-Loubes-Maza-07}. Assume also that for all $t\in{\dR}$, $\varphi_{t}$ is invertible on an interval which will be made precise in the next section. Then, one can consider the random variable $Z_{n}(t)$ defined as
\begin{equation}
\label{defZ}
Z_{n}(t)=\varphi_{t}^{-1}(X_{n})=\varphi_{t}^{-1}\left(\varphi_{\theta}(\varepsilon_{n})\right).
\end{equation}
We also denote by $Z(t)$ a random variable sharing the same distribution as $Z_{n}(t)$.
In order to estimate $\theta$, we choose to evaluate the $L^2$ distance between $\varepsilon$ and $Z(t)$ which is given by
\begin{equation}
\label{defM}
M(t)=\E \left[ \left\vert Z(t) - \varepsilon\right\vert^ 2 \right].
\end{equation}
\\
Denoting $F^{-1}$ the quantile function associated with $\varepsilon$, it can be rewritten as 
$$
M(t)=\E \left[ \left\vert \varphi_{t}^{-1}\left(\varphi_{\theta}(\varepsilon)\right) - \varepsilon\right\vert^ 2 \right]=\int_{0}^{1}\left(\varphi_{t}^{-1}\circ\varphi_{\theta}\circ F^{-1}(x)-F^{-1}(x)\right)^2dx.$$
Indeed (see for instance in \cite{van2000asymptotic} p.305) it is well-known that if $Y$ is a random variable with distribution function $G$, then for $U \sim \mathcal{U}_{\left[0;1 \right]}$, $Y \sim G^{-1} \left( U \right)$. 

 Moreover, if we assume that for all $t$, $\varphi_t$ is increasing, then one have the following expression for the quantile function associated with ${Z(t)}$: $F_{Z(t)}^{-1}=\varphi_{t}^{-1}\circ\varphi_{\theta}\circ F^{-1}$ and so
$$M(t)=\int_{0}^{1}\left( F_{Z(t)}^{-1}(x)-F^{-1}(x)\right)^2dx.$$

This quantity corresponds to the Wasserstein distance between the laws of ${Z(t)}$ and 
$\varepsilon$,  defined and studied for instance in \cite{cuesta1989notes} in general case. Using Wasserstein metrics to align distributions is rather natural since it corresponds to the transportation cost between two probability laws. It is also a proper criterion to  study similarities between point distributions (see for instance in \cite{ MR1625620}) which is already used for density registration in \cite{lescornel:hal-00749519_1} or \cite{ GALLON-2011-593476} in a non sequential way. \\
Hence, in this setting, considering the $L^2$ distance between  the starting point and the registered point is equivalent  to investigate the Wasserstein distance  between their laws. \vskip .1in

As $M(\theta)=0$ and the function $M$ defined by  \eqref{defM} is non-negative, it is clear that $M$ admits at least a global minimum at $\theta$ which permits to have a characterization of the parameter of interest.
%%%%%%%%%%%%%%%%%%%%%%%%%%%%%%%%%%%%%%%%%%%%%%%%%%%%%%%%%%%%%%%%%%%%%%%%%%%%%%%%

\section{ESTIMATION OF THE PARAMETER $\theta$}

%%%%%%%%%%%%%%%%%%%%%%%%%%%%%%%%%%%%%%%%%%%%%%%%%%%%%%%%%%%%%%%%%%%%%%%%%%%%%%%%
In this section, we focus our attention on the estimation of the parameter $\theta\in{\Theta}$ where $\Theta$ is supposed to be an interval of $\mathbb{R}$. Before implementing the estimation procedure for $\theta$,  several hypothesis on the model \eqref{model} are required. 
\begin{itemize}
%\item[ ]
%\begin{equation}
%\label{a0} \Theta \text{ is an interval.} \tag{A0}
%\end{equation}
\item[ ]\begin{equation}
\label{a1} \text{For all } t \in\Theta,~\varphi_t \text{ is  invertible, increasing from } I_{1}\text{ to } I_{2},\text{ some subsets of } \dR.  \tag{A1}
\end{equation}

\item[] \begin{align}
\label{a2}
\tag{A2}
&\text{For all } x \in I_{2},~\varphi^{- 1}_t (x) \text{ is continuously differentiable with respect to }t \in \Theta.\notag\\
\notag&\text{Its derivative is denoted by }\partial \varphi^{- 1}_t(x).
\end{align}

\item[] \begin{equation}
\label{a3}
\text{For all } t \in\Theta,~~%\varphi^{- 1}_t \in L^2 \left(X\right) \text{, that is }
\varphi^{- 1}_t\circ \varphi_{\theta} \in L^2 \left(\varepsilon\right). \tag{A3}
\end{equation}

\item[] \begin{equation}
\label{a4}
\text{ For all compact $B$ in }\Theta,~~\dE \left[ \sup_{t \in B } \left\vert\partial \varphi^{- 1}_t\circ \varphi_{\theta}\left( \varepsilon \right) \right\vert^4\right] < +\infty.
\tag{A4}
\end{equation}

\end{itemize}

From assumption \eqref{a1}, the distribution function of $X$ is $F_{X}=F\circ\varphi_{\theta }^{-1}$ whereas that of $Z(t)$ is $F\circ\varphi_{\theta }^{-1}\circ\varphi_{t}$.

\begin{lem}
\label{MC1}
Assume \eqref{a1} to \eqref{a4}. Then $M$ is continuously differentiable on $\Theta$.
\end{lem}
Using Lemma~\ref{MC1},  the differential $M^{\prime}$ of $M$ has the following expression for all $t\in{\Theta}$,
\begin{align}
\label{defMprime}
M^{\prime}(t)=& -2\int_{0}^{1}\partial\varphi_{t}^{-1}\circ\varphi_{\theta}\circ F^{-1}(x)\left(F^{-1}(x)-\varphi_{t}^{-1}\circ\varphi_{\theta}\circ F^{-1}(x)\right) dx \notag \\% =& -2\E \left[ \partial\varphi_{t}^{-1}\circ\varphi_{\theta}(\varepsilon)\left(\varepsilon-\varphi_{t}^{-1}\circ\varphi_{\theta}(\varepsilon)\right)\right]\notag\\ 
=& -2\E \left[ \partial\varphi_{t}^{-1}\left(X\right)\left(\varepsilon-\varphi_{t}^{-1}\left(X\right)\right)\right]
.
\end{align}
It is then clear that $M^{\prime}(\theta)=0$. Then, we can assume that there exists $\{a,b\}\in{\Theta^{2}}$ with $a<b$ and $\theta\in{]a;b[} \subset\Theta$ such that, for all $t\in{[a;b]}$,
\begin{equation}
\label{a5}
(t-\theta)M^{\prime}(t)>0.
\tag{A5}
\end{equation}
We are now in position to implement our Robbins-Monro procedure. More precisely, denote by $\pi_{[a;b]}$ the projection on the compact set $[a;b]$ defined for all $x\in{[a;b]}$ by
%\begin{eqnarray*}
%   \pi_{[a;b]}(x) = \left \{ \begin{array}{lll}
%   \,\,x & \ \text{ if } \ a\leq x\leq b, \vspace{1ex} \\
%    \,\,a & \ \text{ if } \ x\leq a,  \vspace{1ex} \\
%    \,\, b & \ \text{ if } \ x\geq b. \\
%   \end{array} \nonumber \right.
%\end{eqnarray*}
$$
\pi_{[a;b]}(x)=x\rI_{\{a\leq{x}\leq{b}\}}+a\rI_{\{x\leq{a}\}}+b\rI_{\{x\geq{b}\}}.
$$
Let $(\gamma_n)$ be a decreasing sequence of positive real numbers satisfying
\begin{equation}
\label{hypgamma}
\sum_{n=1}^\infty\gamma_{n}=+\infty
\hspace{1cm}\text{and}\hspace{1cm}
\sum_{n=1}^\infty\gamma_{n}^2<+\infty.
\end{equation}
We estimate the parameter $\theta$ via the projected Robbins-Monro algorithm
\begin{equation}
\label{RMA}
\wh{\theta}_{n+1}=\pi_{[a;b]}\Bigl(\wh{\theta}_{n}-\gamma_{n+1}T_{n+1}\Bigr)
\end{equation}
where the deterministic initial value $\wh{\theta}_{0} \in [a;b]$ and the random variable $T_{n+1}$ is defined by
\begin{equation}
\label{DefT}
T_{n+1}=-2\partial\varphi_{\wh{\theta}_{n}}^{-1}\left(X_{n+1}\right)\left(\varepsilon_{n+1}-\varphi_{\wh{\theta}_{n}}^{-1}\left(X_{n+1}\right)\right).
\end{equation}

Our results of convergence for the estimator $\wh{\theta}_{n}$ are as follows.

\begin{thm}
\label{ASCV}
Assume \eqref{a1} to \eqref{a5}, with $\theta\in{]a;b[}$ where $a<b$. Then, $\wh{\theta}_{n}$ converges almost surely to $\theta$.
\end{thm}

In order to get a control on the rate of convergence of $\wh{\theta}_{n}$ towards $\theta$, we need  to assume the following  slightly stronger condition of regularity on the deformation functions.
\begin{itemize}
\item[] \begin{align}
\label{a6}
\tag{A6}
&\text{For all }x \in I_{2},~~\varphi^{- 1}_t (x) \text{ is twice differentiable with respect to }t \in \Theta \notag
\\ \notag&\text{and
for all compact $B$ in } \Theta,~~\dE \left[ \sup_{t \in B } \left\vert\partial^2 \varphi^{- 1}_t\circ \varphi_{\theta}\left( \varepsilon \right) \right\vert^2\right] < +\infty.
\end{align}
\end{itemize}

\begin{lem}
\label{MC2}
Assume \eqref{a1} to \eqref{a6}. Then $M$ is twice continuously differentiable on $\Theta$.
\end{lem} 
Then we can compute the second differential of $M^{\prime\prime}$ of $M$  for all $t\in{\Theta}$ as
\begin{align}
\label{defMsecond}
M^{\prime\prime}(t)=& 2\int_{0}^{1}\left[ \partial\varphi_{t}^{-1}\circ\varphi_{\theta}\circ F^{-1}(x)\right]^2dx \\ \notag &- 2\int_{0}^{1}\partial^2\varphi_{t}^{-1}\circ\varphi_{\theta}\circ F^{-1}(x)\left(F^{-1}(x)-\varphi_{t}^{-1}\circ\varphi_{\theta}\circ F^{-1}(x)\right) dx
\end{align}
that is 
%\begin{align}
%%\label{defMsecond}
%M^{\prime\prime}(t)=& 2\E\left[ \left( \partial\varphi_{t}^{-1}\circ\varphi_{\theta}(\varepsilon)\right)^2 \right]  - 2\E\left[\partial^2\varphi_{t}^{-1}\circ\varphi_{\theta}(\varepsilon)\left(\varepsilon-\varphi_{t}^{-1}\circ\varphi_{\theta}(\varepsilon)\right)\right] .
%\end{align}

\begin{align}
%\label{defMsecond}
M^{\prime\prime}(t)=& 2\E\left[ \left( \partial\varphi_{t}^{-1}(X)\right)^2 \right]  - 2\E\left[\partial^2\varphi_{t}^{-1}(X)\left(\varepsilon-\varphi_{t}^{-1}(X)\right)\right] .
\end{align}
For the sake of clarity, we shall make use of $\gamma_n=1/n$ for the following theorem.
\begin{thm}
\label{TLC}
Assume \eqref{a1} to \eqref{a6}, with $\theta\in{]a;b[}$ where $a<b$. In addition, suppose that $M^{\prime\prime}(\theta)>1/2$ and that there exists $\alpha>4$ such that for all compact $B$ in $\Theta$,
$$
\mathbb{E}\left[\underset{t\in{B}}\sup|\partial\varphi_{t}^{-1}\circ\varphi_{\theta}(\varepsilon)|^{\alpha}\right]<+\infty.
$$
Then, we have as $n$ goes to infinity, the degenerated asymptotic normality
\begin{equation}
\label{dtlc}
\sqrt{n}\left(\wh{\theta}_{n}-\theta\right)\liml \delta_{0}.
\end{equation}
Moreover, if for all $t\in{[a;b]}$, 
\begin{equation}
\label{a7}
\tag{A7}
M^{\prime\prime}(t)\geq1/2,
\end{equation}
then for all $n\geq0$,
\begin{equation}
\label{mean2theta}
\mathbb{E}\left[\left(\wh{\theta}_{n}-\theta\right)^{2}\right]\leq{\left(\wh{\theta}_{0}-\theta\right)^{2}\frac{\exp\left(C_{1}\pi^{2}/6\right)}{n+1}}
\end{equation}
where 
\begin{equation}
\label{defC1}
C_{1}=4\mathbb{E}\left[\underset{t\in{[a;b]}}\sup|\partial\varphi_{t}^{-1}\circ\varphi_{\theta}(\varepsilon)|^{4}\right].
\end{equation}
\end{thm}

\begin{proof}
The proofs are postponed to Section 6.
\end{proof}

\begin{rem}
One can observe that $$M^{\prime\prime}(\theta)=2\int_{0}^{1}\left[ \partial\varphi_{\theta}^{-1}\circ\varphi_{\theta}\circ F^{-1}(x)\right]^2dx =2\E\left[ \left( \partial\varphi_{\theta}^{-1}(X)\right)^2 \right].$$
%$$M^{\prime\prime}(\theta)=22\E\left[ \left( \partial\varphi_{t}^{-1}\circ\varphi_{\theta}(\varepsilon)\right)^2 \right] $$
Hence the inequality $M^{\prime\prime}(\theta)>0$ holds in the general case. Moreover, replacing $M$ by $\lambda M$ where $\lambda$ is a real and positive number does not change any results. Then, the condition $M^{\prime\prime}(t)\geq1/2$ may be verified with little modifications. 
\end{rem}

\begin{rem}
From a theoretical point of view, it could be interesting to obtain a non-degenerated asymptotic normality than the one obtained in \eqref{dtlc}. For that purpose, one consider a slight modification of the algorithm defined by \eqref{RMA}. More precisely, it consists in replacing the algorithm \eqref{RMA} by its \og{excited}\fg\hspace{1mm}version
\begin{equation}
\label{RMAmodif}
\widetilde{\theta}_{n+1}=\pi_{[a;b]}\Bigl(\widetilde{\theta}_{n}-\gamma_{n+1}\widetilde{T}_{n+1}\Bigr)
\end{equation}
where the initial deterministic value $\widetilde{\theta}_{0} \in [a;b]$ and the random variable $\widetilde{T}_{n+1}$ is defined by
\begin{equation}
\label{DefTmodif}
\widetilde{T}_{n+1}=-2\partial\varphi_{\widetilde{\theta}_{n}}^{-1}\left(X_{n+1}\right)\left(\varepsilon_{n+1}-\varphi_{\widetilde{\theta}_{n}}^{-1}\left(X_{n+1}\right)\right)+V_{n+1}
\end{equation}
where $\left(V_{n}\right)$ is a sequence of independent and identically distributed simulated random variables with mean $0$ and variance $\sigma^{2}>0$. Then, thanks to this persistent excitation, Theorem \ref{ASCV} and Theorem \ref{TLC} are still true for $\widetilde{\theta}_{n}$ where \eqref{dtlc} is replaced by
\begin{equation}
\label{ndtlc}
\sqrt{n}\left(\widetilde{\theta}_{n}-\theta\right)\liml \mathcal{N}\left(0,\frac{\sigma^{2}}{2M^{\prime\prime}(\theta)-1}\right).
\end{equation}
\end{rem}

%%%%%%%%%%%%%%%%%%%%%%%%%%%%%%%%%%%%%%%%%%%%%%%%%%%%%%%%%%%%%%%%%%%%%%%%%%%%%%%%

\section{ESTIMATION OF THE DENSITY}

%%%%%%%%%%%%%%%%%%%%%%%%%%%%%%%%%%%%%%%%%%%%%%%%%%%%%%%%%%%%%%%%%%%%%%%%%%%%%%%%

In this section, we suppose that the random variable $\varepsilon$ has a density $f$ and we focus on the non-parametric estimation of this density. 
A natural way to estimate $f$ is to consider the recursive Parzen-Rosenblatt estimator defined for all $x\in{I_{1}}$, by
\begin{equation}
\label{PRstandard}
\widetilde{f}_{n}(x)=\frac{1}{n} \sum_{i=1}^n \frac{1}{h_i}K \left( \frac{x-\varepsilon_{i}}{h_i} \right).
\end{equation}
where $K$ is a standard kernel function.
It is well known that $\widetilde{f}_{n}$ is a really good approximation of $f$ for large values of $n$. However, for small samples corresponding to small values of $n$, $\widetilde{f}_{n}$ may not be a good estimator of  $f$. Hence, it could be interesting to have more realizations of $\varepsilon$ in order to get a better approximation. In our case, we know that $Z_{n}(\theta)=\varepsilon_{n}$. Then, the idea is to use the prior estimation of $\theta$ in order to construct a Parzen-Rosenblatt estimator of $f$ which will be of length $2n$.
Further assumptions must to be added to hypothesis \eqref{a1} to \eqref{a6}. More precisely, if we denote by $\partial$ the differential operator with respect to $t$ and $d$ the differential operator with respect to $x$, we need the following hypothesis on the regularity of $f$ and on the deformation functions $\varphi_{t}$.
\begin{itemize}
\item[ ]
\begin{align}
\label{ad1}  \tag{AD1}
&f\text{ is bounded, twice continuously differentiable on }I_{1},\\
\notag
&\text{with bounded derivatives.}
 \end{align}
 
\item[ ] \begin{equation}
\label{ad2}\tag{AD2} 
\text{For all } t\in{\Theta}, \varphi_t  \text{ is three times continuously differentiable on }  I_{1}.
 \end{equation}

\item[ ] \begin{align}
\label{ad3} \tag{AD3} 
&\varphi^{-1}_{\theta}  \text{ is three times continuously differentiable on } I_{2},\\
\notag
&\text{with bounded derivatives.}
\end{align}

\item[ ]  \begin{equation}
\label{ad4} \tag{AD4} 
d\varphi, d^2\varphi,  d^3\varphi  \text{ are bounded}.
\end{equation}
\end{itemize}

Denote by $K$ a positive kernel which is a symmetric, integrable and bounded function, such that 
$$
\int_\R K(u)du =1,\hspace{5mm}\textnormal{}\hspace{5mm}\lim_{\left\vert x\right\vert \rightarrow +\infty } \left\vert x \right\vert K(x) = 0,\hspace{5mm}\textnormal{and}\hspace{5mm}\int_\R u^2K(u)du <+\infty.
$$
Then we consider the following recursive estimate
\begin{equation}
\label{PRestimate}
\widehat{f}_n(x) =\frac{1}{n} \sum_{i=1}^n \frac{1}{h_i}K \left( \frac{x-{Z}_i(\widehat{\theta}_{i-1})}{h_i} \right),
\end{equation}
where $\widehat{\theta}_{i-1}$ is given by \eqref{RMA} and where the bandwidth $\left( h_n\right)$ is a sequence of positive real numbers, decreasing to zero, such that  $n h_n$ tends to infinity when $n$ goes to infinity. For sake of simplicity, we make use of $h_n=\frac{1}{n^ \alpha}$ with $0<\alpha<1$.
The following result deals with the pointwise almost sure convergence of $\widehat{f}_n(x)$.
\begin{thm}
\label{th:cvps}
Assume \eqref{a1} to \eqref{a5} with  $\theta\in{]a;b[}$ where $a<b$ and \eqref{ad1} to \eqref{ad4}. Then for all $x\in I_{1}$, 
\begin{equation}
\label{ASCVdensity}
\widehat{f}_n(x)\xrightarrow{n \rightarrow \infty} f(x) \hspace{6mm}\textnormal{ a.s.}
\end{equation}
\end{thm}
It follows from Theorem \ref{th:cvps} that for small values of $n$, the averaged estimator 
$$
\bar{f}_{n}=\frac{1}{2}\left(\widetilde{f}_{n}+\widehat{f}_{n}\right)
$$
where $\widetilde{f}_{n}$ and $\widehat{f}_{n}$ are given by \eqref{PRstandard} and \eqref{PRestimate}, will perform better than $\widetilde{f}_{n}$ or $\widehat{f}_{n}$.\\
\\
The second result of this section concerns the convergence in quadratic mean of $\widehat{f}_n(x)$ to $f(x)$. In this way, we need to add to hypothesis \eqref{ad1} to \eqref{ad4} the following little stronger assumption on the regularity of the deformation functions $\varphi$.
\begin{itemize}
\item[ ] \begin{align}
\label{ad5}\tag{AD5} 
&\varphi \text{ is twice continuously differentiable on }\Theta \times I_{1}\\
 \notag 
&\text{and } \partial \varphi_t(x), \partial d\varphi_t(x)  \text{ are bounded with respect to }t. \end{align}
\end{itemize}
\begin{thm}
\label{th:cvl2}
Assume \eqref{a1} to \eqref{a7} with $\theta\in{]a;b[}$ where $a<b$ and \eqref{ad1} to \eqref{ad5}. Then, for all $x\in I_{1}$,
\begin{equation}
\label{meanCVdensity}
\E \left[ \left\vert \widehat{f}_n(x)-f(x)\right\vert^2 \right] \xrightarrow{n \rightarrow \infty} 0.
\end{equation}
\end{thm}
\begin{proof}
The proofs are postponed to Section 7.
\end{proof}

%%%%%%%%%%%%%%%%%%%%%%%%%%%%%%%%%%%%%%%%%%%%%%%%%%%%%%%%%%%%%%%%%%%%%%%%%%%%%%%%

\section{SIMULATIONS}

%%%%%%%%%%%%%%%%%%%%%%%%%%%%%%%%%%%%%%%%%%%%%%%%%%%%%%%%%%%%%%%%%%%%%%%%%%%%%%%%

This section is devoted to the numerical illustration of the asymptotic properties of our estimator $\widehat{\theta}_{n}$ defined by \eqref{RMA}. Note that for the model \eqref{model}, the transformations $\varphi_{\theta}$ which are inversible with respect to $\theta$ have no great interest because, in this case, it is possible to express $\theta$ in terms of $X_{0},\hdots,X_{n},\varepsilon_{0},\hdots,\varepsilon_{n}$. However, when $\varphi_{\theta}$ is not invertible with respect to $\theta$, it is not possible to use a direct expression for the estimator and 
%write $\theta$ in terms of $X_{0},\hdots,X_{n},\varepsilon_{0},\hdots,\varepsilon_{n}$, and 
our procedure is useful in order to estimate $\theta$. Among the many transformations of interest, we focus  here on two of them that are used in econometry. More precisely, we illustrate our estimation procedure for the Box-Cox transformation $\varphi_{t}^{1}$ and the arcsinh transformation $\varphi_{t}^{2}$. The transformation $\varphi_{t}^{1}$ is given, for all $x\in{\mathbb{R}^{+}_{*}}$, by
\begin{eqnarray}
\label{BoxCox}
   \varphi_{t}^{1}(x) = \left \{ \begin{array}{lll}
   \displaystyle\frac{x^{t}-1}{t} & \ \text{ if }& \ t\neq 0\vspace{1ex} \\
       \hspace{2mm}\log(x) & \ \text{ if } &\ t=0\\
   \end{array}\right.
\end{eqnarray}
whereas $\varphi_{t}^{2}$ is given for all $x\in{\mathbb{R}}$, by
\begin{eqnarray}
\label{arcshtransfo}
   \varphi_{t}^{2}(x) = \left \{ \begin{array}{lll}
   \displaystyle\frac{1}{t}\sinh^{-1}(t x) & \ \text{ if }& \ t\neq 0\vspace{1ex} \\
       \hspace{5mm}x & \ \text{ if } &\ t=0.\\
   \end{array}\right.
\end{eqnarray}
Throughout this section, we suppose that $\theta>0$, and specifically we assume that $\theta\in{]a;b[}$ with $a=1/10$ and $b=2$. Then, the Box-Cox transform $\varphi_{t}^{1}$ is invertible from $]1;+\infty[$ to $\mathbb{R}^{+}_{*}$ and the arcsinh transformation is invertible from $\mathbb{R}$ to $\mathbb{R}$. Moreover, the inverses $\left(\varphi_{t}^{1}\right)^{-1}$ and $\left(\varphi_{t}^{2}\right)^{-1}$ of $\varphi_{t}^{1}$ and $\varphi_{t}^{2}$ are given by
\begin{equation}
\label{inverseBoxCox}
\forall x\in{\mathbb{R}^{+}_{*}},\hspace{1cm}\left(\varphi_{t}^{1}\right)^{-1}(x)=\left(1+t x\right)^{1/t}
\end{equation}
and
\begin{equation}
\label{inversearcsh}
\forall x\in{\mathbb{R}},\hspace{1cm}\left(\varphi_{t}^{2}\right)^{-1}(x)=\frac{1}{t}\sinh(t x).
\end{equation}
Hence, it is clear that for all $t\in{[a;b]}$, $\left(\varphi_{t}^{1}\right)^{-1}(x)$ and $\left(\varphi_{t}^{2}\right)^{-1}(x)$ are continuously differentiable according to $t$ and that
\begin{equation}
\label{DinverseBoxCox}
\forall x\in{\mathbb{R}^{+}_{*}},\hspace{1cm}\partial\left(\varphi_{t}^{1}\right)^{-1}(x)=\frac{1}{t}\left(\frac{x}{1+t x}-\frac{1}{t}\log(1+t x)\right)\left(1+t x\right)^{1/t}
\end{equation}
and
\begin{equation}
\label{Dinversearcsh}
\forall x\in{\mathbb{R}},\hspace{1cm}\partial\left(\varphi_{t}^{2}\right)^{-1}(x)=-\frac{1}{t}\left(\frac{1}{t}\sinh(t x)-x\cosh(t x)\right).
\end{equation}
Denote by $M^{1}$, respectively $M^{2}$, the function $M$ given by \eqref{defM} associated with $\varphi_{t}^{1}$ and $\varphi_{t}^{2}$. For the simulations, we choose $\theta=1$. The functions $M^{1}$ and $M^{2}$ are represented in Figure \ref{functionM}. One can see that $\theta$ is effectively a global minimum of $M^{1}$ and $M^{2}$.
\begin{figure}[htb]
\vspace{-2ex}
\begin{minipage}[b] {0.48\linewidth}
\centering 
\centerline {\includegraphics[scale=0.45]{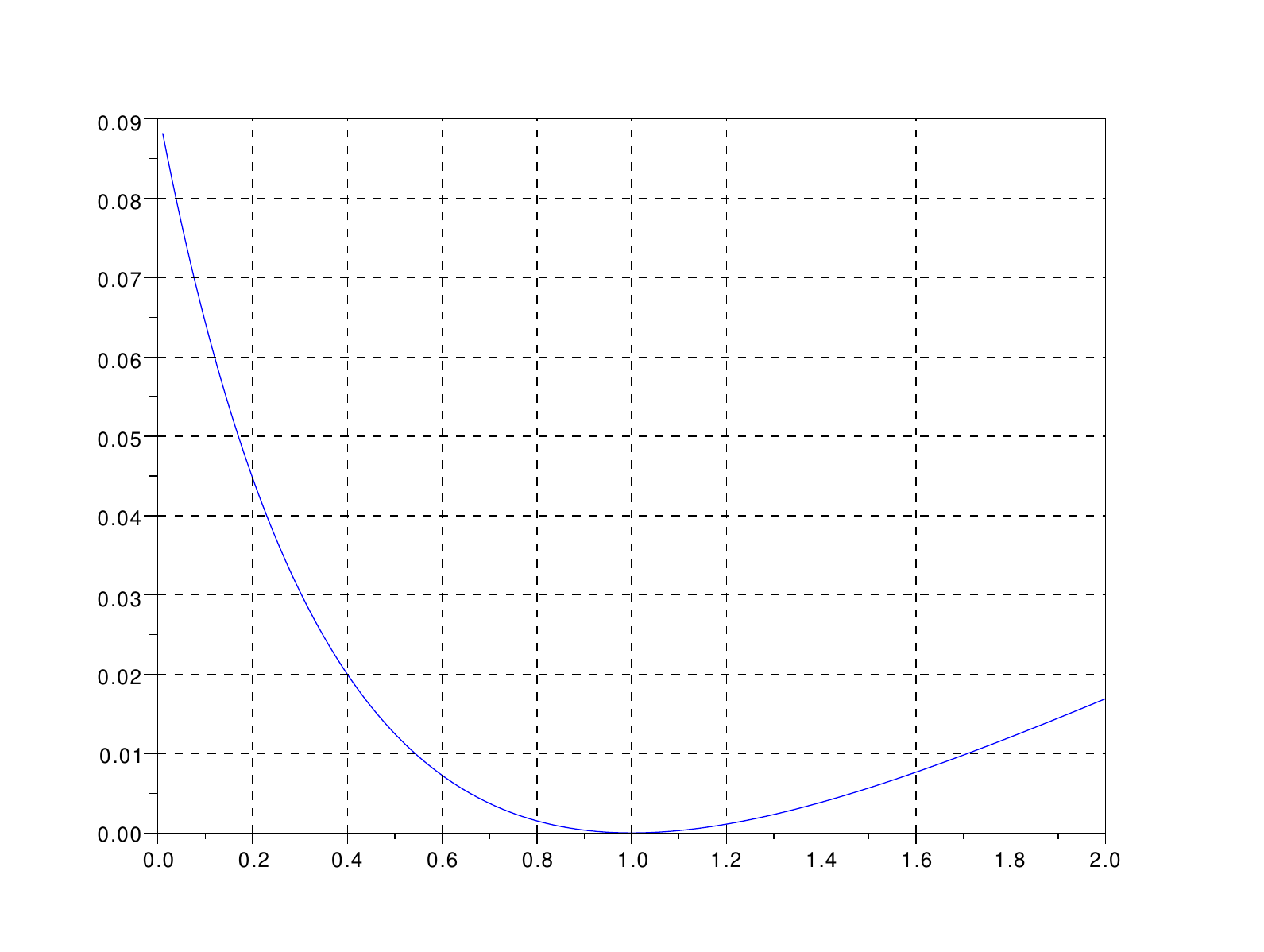} }
\vspace{0.1cm}
\medskip
\end{minipage}
\hfill
\begin{minipage}[b]{0.48\linewidth}
\centering 
\centerline {\includegraphics[scale=0.45]{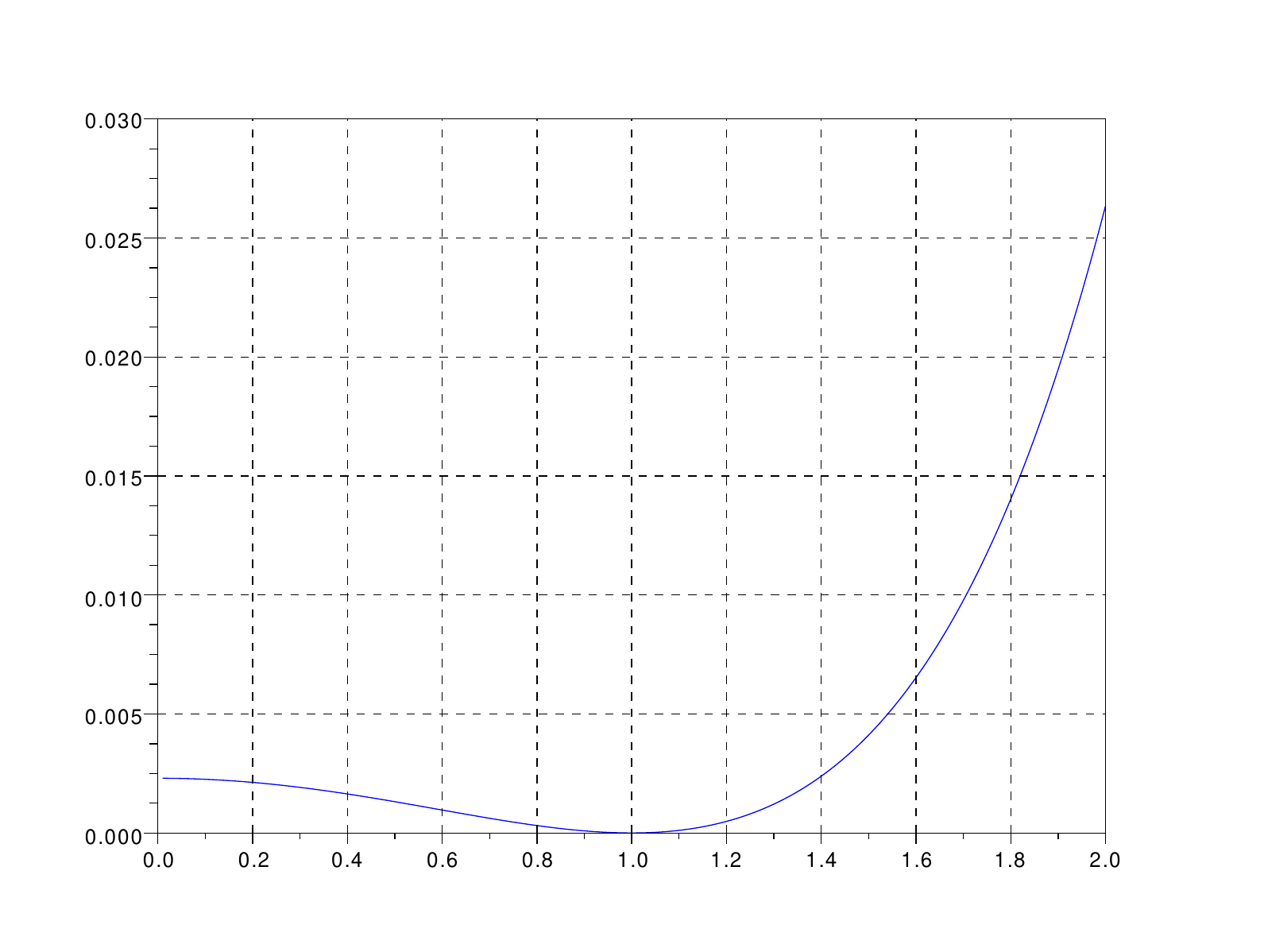} }
\vspace{0.1cm}
\medskip
\end{minipage}
\vspace{-2ex}
\caption{The functions $M^{1}$ and $M^{2}$}
\label{functionM}
\end{figure}
For the estimation of $\theta$ in both models, one chooses $\left(\varepsilon_{n}^{1}\right)$ a sequence of independent random variables whose distribution is uniform on $[0;1]$ and $\left(\varepsilon_{n}^{2}\right)$ a sequence of independent random variables whose distribution is uniform on $[1;2]$. We simulate random variables $X_{n}^{1}$ and $X_{n}^{2}$ according to the model \eqref{model}
$$
X_{n}^{i}=\varphi_{\theta}^{i}\left(\varepsilon_{n}^{i}\right),
$$ 
for $i=1,2$. Then, for $i=1,2$ and for the choice of step $\gamma_{n}=1/n$, we compute the sequence $\wh{\theta}_{n}^{i}$ according to \eqref{RMA}. More precisely,
$$
\wh{\theta}_{n+1}^{i}=\pi_{[a;b]}\left(\wh{\theta}_{n}^{i}-\gamma_{n}T_{n+1}^{i}\right)
$$
where 
$$
T_{n+1}^{i}=-2\partial\left(\varphi_{\wh{\theta}_{n}^{i}}^{i}\right)^{-1}(X_{n+1}^{i})\left(\varepsilon_{n+1}^{i}-\left(\varphi_{\wh{\theta}_{n}^{i}}^{i}\right)^{-1}(X_{n+1}^{i})\right),
$$
and where $(\varphi_{\wh{\theta}_{n}^{i}}^{i})^{-1}$ are given by \eqref{inverseBoxCox} and \eqref{inversearcsh} and $\partial(\varphi_{\wh{\theta}_{n}^{i}}^{i})^{-1}$ are given by \eqref{DinverseBoxCox} and \eqref{Dinversearcsh}. The values of $\wh{\theta}_{n}^{i}$ are computed until $n=1000$. We represent on the left-hand side (respectively on the right-hand side) of Figure \ref{cvps} the difference between $\wh{\theta}_{n}^{1}$ and $\theta$ (respectively $\wh{\theta}_{n}^{2}$ and $\theta$) for $1\leq{n}\leq{1000}$. In particular,  we obtain that $|\wh{\theta}_{1000}^{1}-\theta|=0.00239$ and $|\wh{\theta}_{1000}^{2}-\theta|=0.0042$ showing that our procedure performs very well for both models.
In addition, on the left-hand side of Figure \ref{cvloi}, one have represented the degenerated asymptotic normality given by \eqref{dtlc} for the data generated according to the model \eqref{model} associated with $\varphi_{\theta}^{1}$. For that, we have made $200$ realizations of the random variable $\sqrt{1000}\left(\widehat{\theta}_{1000}^{1}-\theta\right)$. Finally, one also consider the excited version \eqref{RMAmodif} of algorithm \eqref{RMA} for the first deformation $\varphi_{\theta}^{1}$
$$
\widetilde{\theta}_{n+1}^{1}=\pi_{[a;b]}\left(\widetilde{\theta}_{n}^{1}-\gamma_{n}\widetilde{T}_{n+1}^{1}\right),
$$
with
$$
\widetilde{T}_{n+1}^{1}=-2\partial\left(\varphi_{\wh{\theta}_{n}^{1}}^{1}\right)^{-1}(X_{n+1}^{1})\left(\varepsilon_{n+1}^{1}-\left(\varphi_{\wh{\theta}_{n}^{1}}^{1}\right)^{-1}(X_{n+1}^{1})\right)+V_{n+1},
$$
 where the sequence $(V_{n})$ is a sequence of independent random variables simulated according to the law $\mathcal{N}\left(0,1/2\right)$. As for the degenerated asymptotic normality, one have made $200$ realizations of the random variable $\sqrt{1000}\left(\widetilde{\theta}_{1000}^{1}-\theta\right)$ in order to illustrate the asymptotic normality given by \eqref{ndtlc}. This last numerical result is represented on the right-hand side of Figure \ref{cvloi}.
\begin{figure}[htb]
\vspace{-2ex}
\begin{minipage}[b] {0.48\linewidth}
\centering 
\centerline {\includegraphics[scale=0.45]{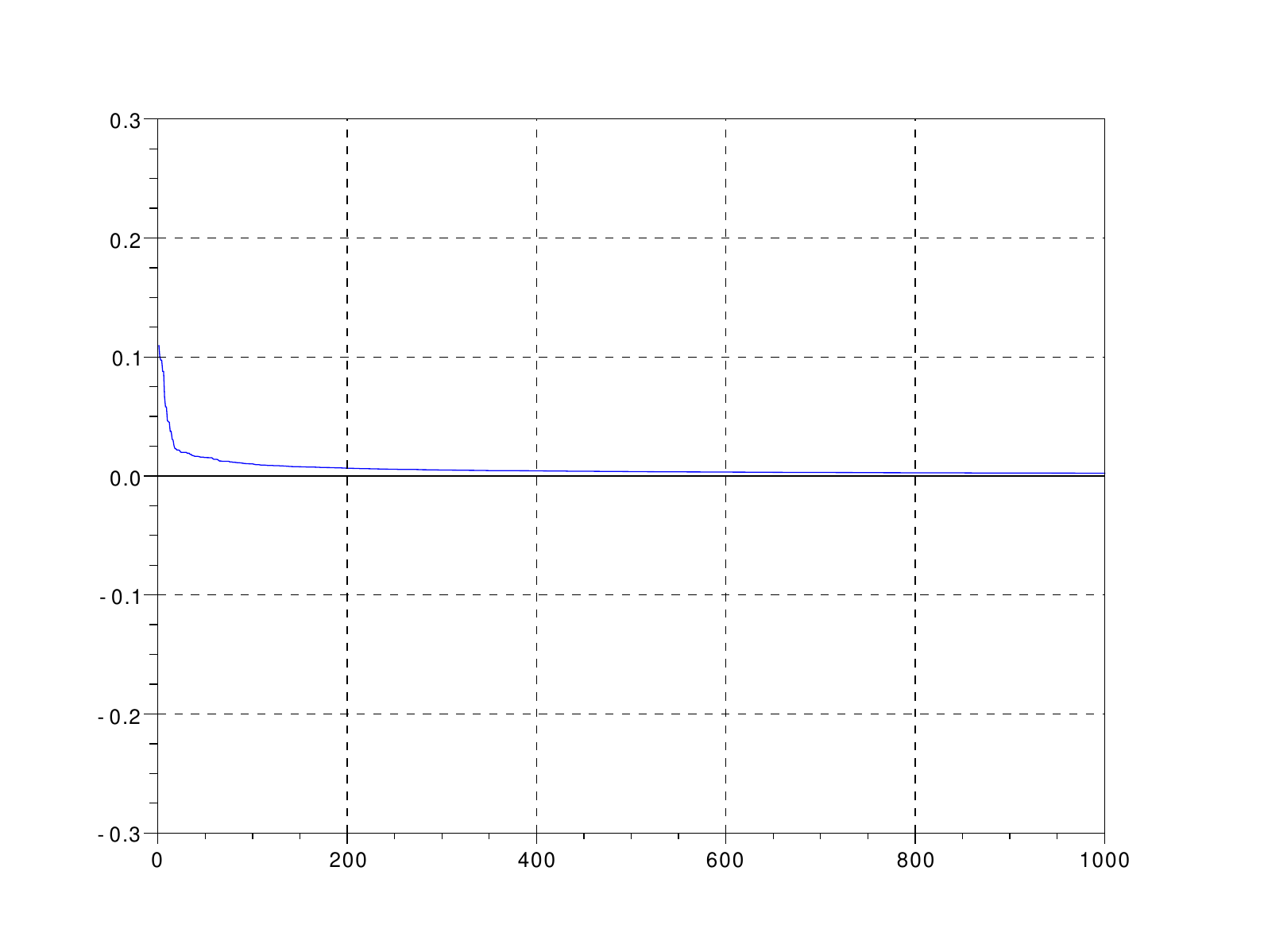} }
\vspace{0.1cm}
\medskip
\end{minipage}
\hfill
\begin{minipage}[b]{0.48\linewidth}
\centering 
\centerline {\includegraphics[scale=0.45]{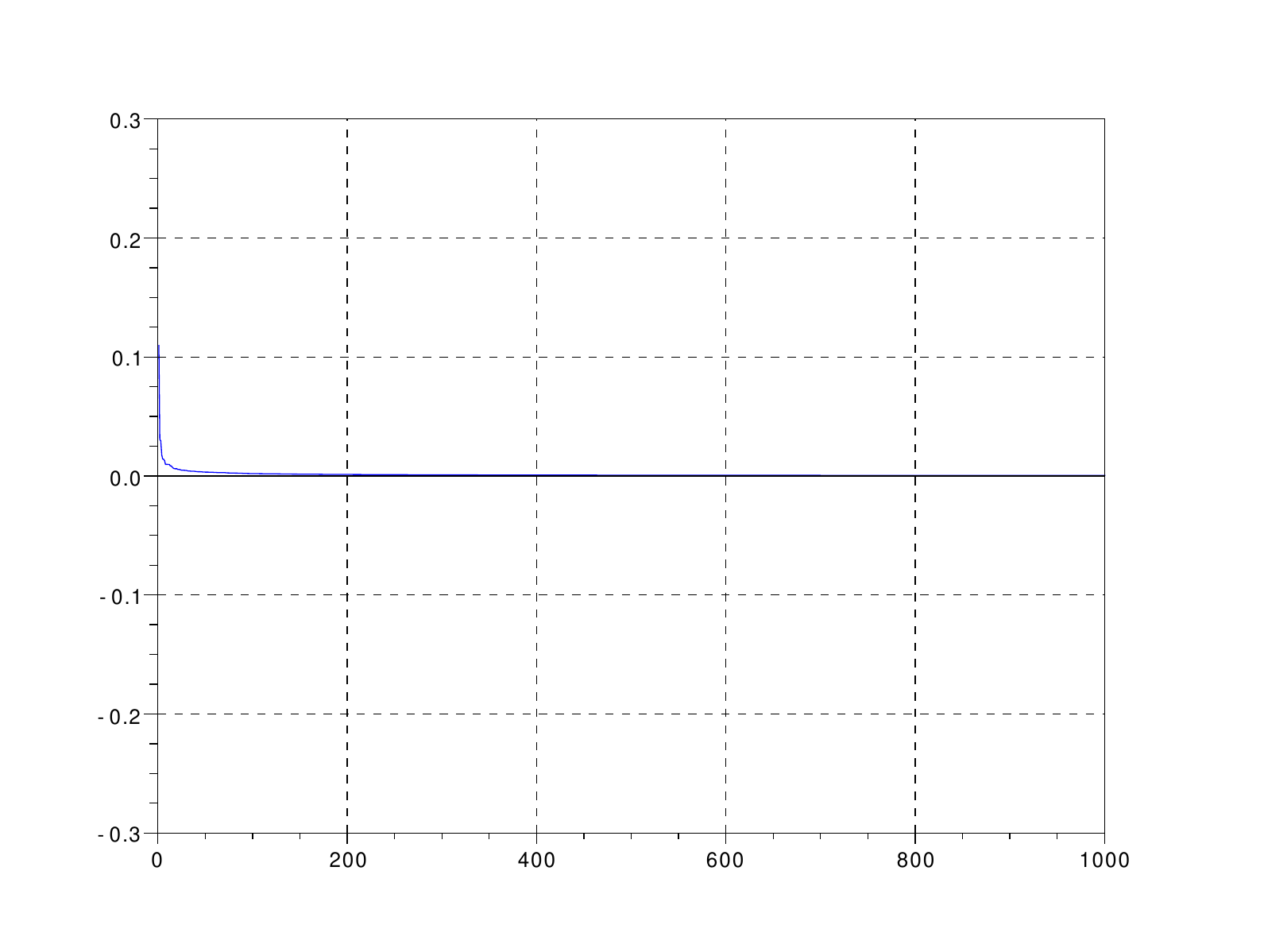} }
\vspace{0.1cm}
\medskip
\end{minipage}
\vspace{-2ex}
\caption{Difference $\wh{\theta}_{n}^{1}-\theta$ and $\wh{\theta}_{n}^{2}-\theta$.}
\label{cvps}
\end{figure}

\begin{figure}[htb]
\vspace{-2ex}
\begin{minipage}[b] {0.48\linewidth}
\centering 
\centerline {\includegraphics[scale=0.45]{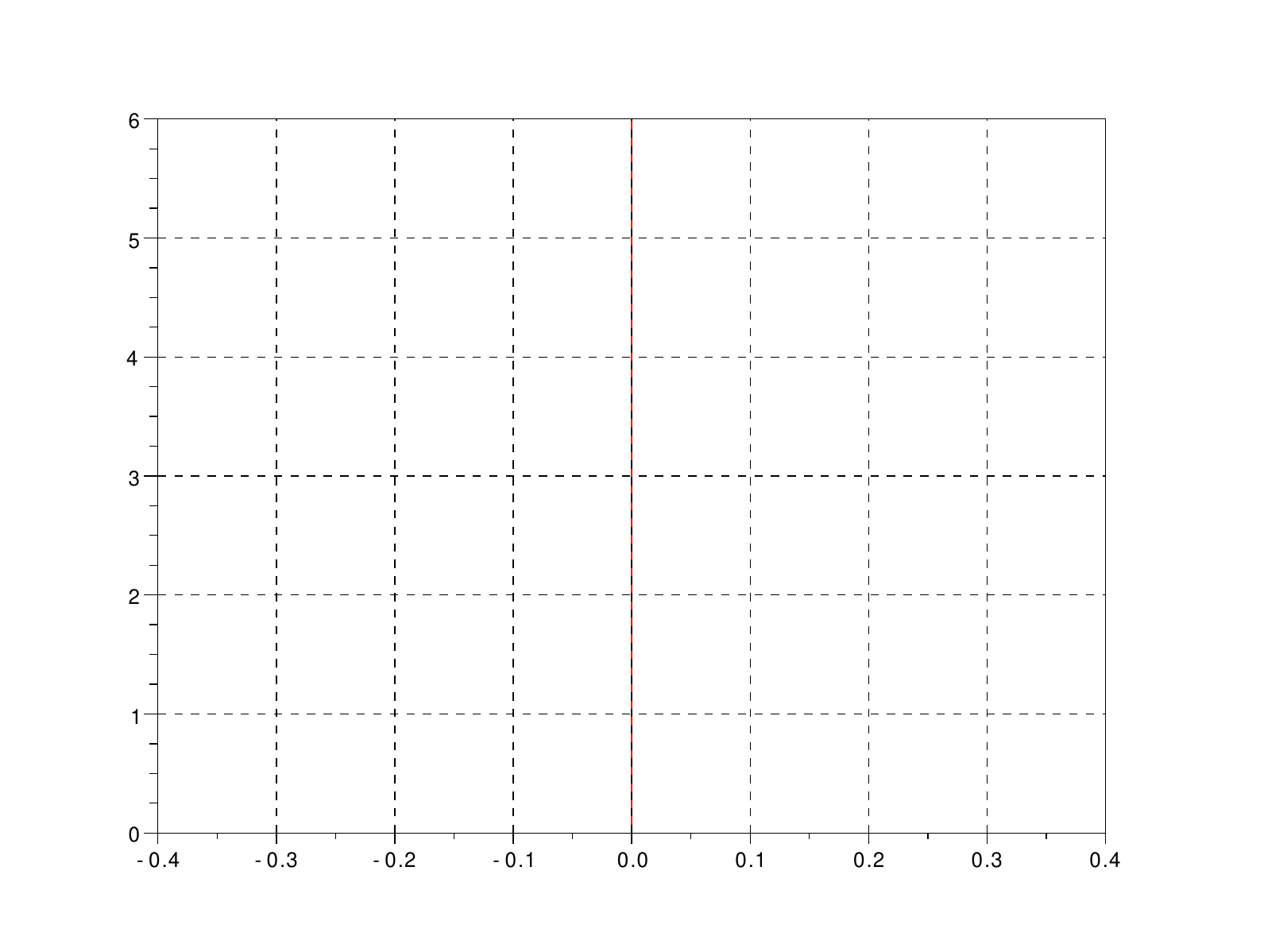} }
\vspace{0.1cm}
\medskip
\end{minipage}
\hfill
\begin{minipage}[b]{0.48\linewidth}
\centering 
\centerline {\includegraphics[scale=0.45]{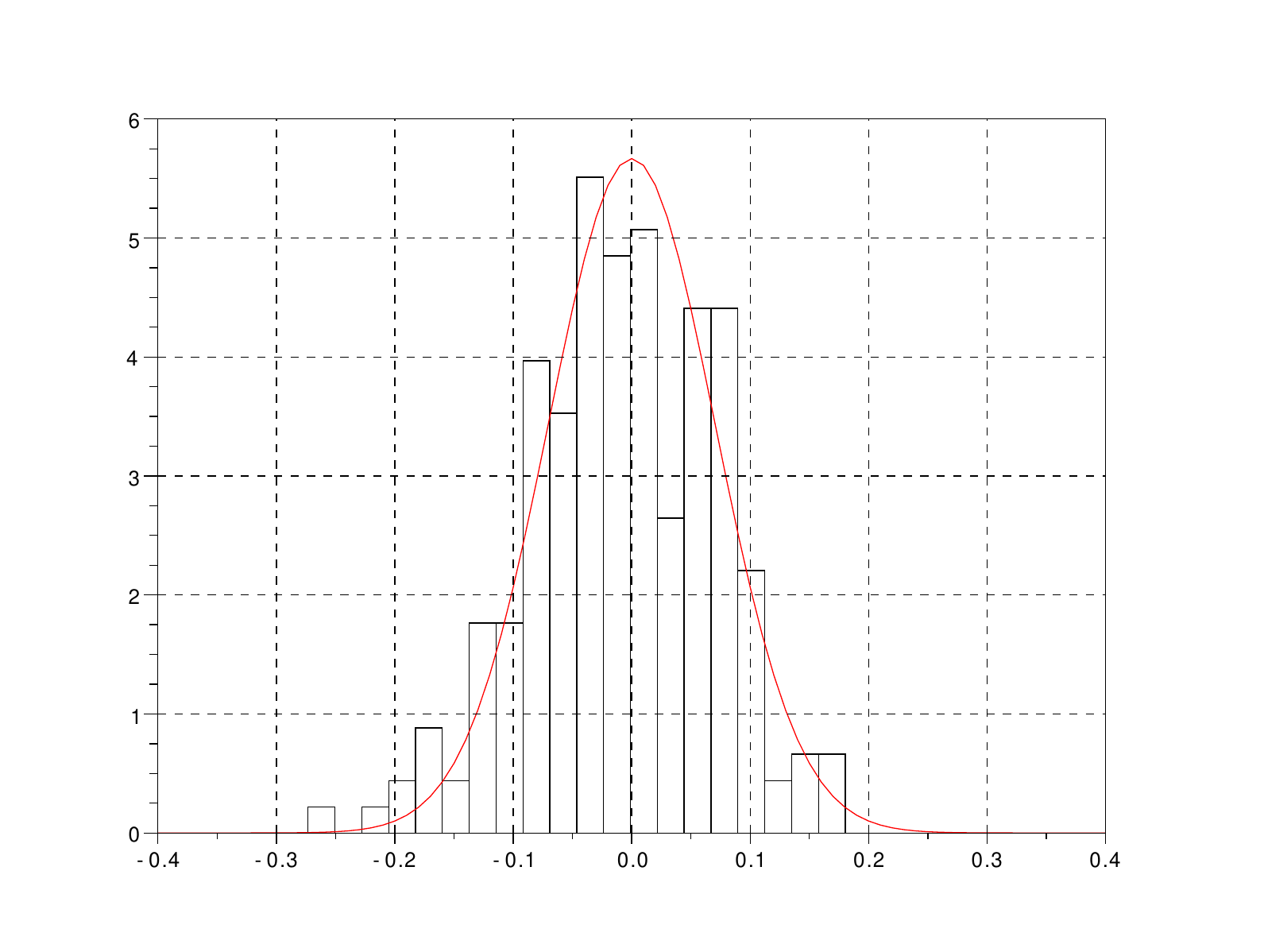} }
\vspace{0.1cm}
\medskip
\end{minipage}
\vspace{-2ex}
\caption{Asymptotic normalities of $\sqrt{n}\left(\widehat{\theta}_{n}^{1}-\theta\right)$ and $\sqrt{n}\left(\widetilde{\theta}_{n}^{1}-\theta\right)$.}
\label{cvloi}
\end{figure}

%\begin{figure}[htb]
%\vspace{-2ex}
%\begin{minipage}[b] {0.48\linewidth}
%\centering 
%\centerline {\includegraphics[scale=0.5]{RM_moment_n=20.pdf} }
%\vspace{0.1cm}
%\medskip
%\end{minipage}
%\hfill
%\begin{minipage}[b]{0.48\linewidth}
%\centering 
%\centerline {\includegraphics[scale=0.5]{RM_moment_n=200.pdf} }
%\vspace{0.1cm}
%\medskip
%\end{minipage}
%\hfill
%\begin{minipage}[b]{0.48\linewidth}
%\centering 
%\centerline {\includegraphics[scale=0.5]{RM_moment_n=2000.pdf} }
%\vspace{0.1cm}
%\medskip
%\end{minipage}
%\vspace{-2ex}
%\caption{Comparison of the convergences of $\wh{\theta}_{n}$ and $\Bar{\theta}_{n}$ for $n=20$, $n=200$ and $n=2000$}
%\label{cvps}
%\end{figure}

%\begin{figure}[htb]
%\vspace{-2ex}
%\begin{minipage}[b] {0.48\linewidth}
%\centering 
%\centerline {\includegraphics[scale=0.4]{cvps1.pdf} }
%\vspace{0.1cm}
%\medskip
%\end{minipage}
%\caption{Almost sure convergence of $\widehat{\theta}_{n}$ and $\widetilde{\theta}_{n}$ to $\theta$}
%\label{asconvergence}
%\end{figure}
%
%
%\begin{figure}[htb]
%\vspace{-2ex}
%\begin{minipage}[b] {0.48\linewidth}
%\centering 
%\centerline {\includegraphics[scale=0.5]{cvdirac3.pdf} }
%\vspace{0.1cm}
%\medskip
%\end{minipage}
%\hfill
%\begin{minipage}[b]{0.48\linewidth}
%\centering 
%\centerline {\includegraphics[scale=0.5]{cvloi3.pdf} }
%\vspace{0.1cm}
%\medskip
%\end{minipage}
%\vspace{-2ex}
%\caption{Asymptotic normalities for $f_{3}$}
%\label{asnormality3}
%\end{figure}

%%%%%%%%%%%%%%%%%%%%%%%%%%%%%%%%%%%%%%%%%%%%%%%%%%%%%%%%%%%%%%%%%%%%%%%%%%%%%%%%

\section{PROOFS OF THE PARAMETRIC RESULTS}

%%%%%%%%%%%%%%%%%%%%%%%%%%%%%%%%%%%%%%%%%%%%%%%%%%%%%%%%%%%%%%%%%%%%%%%%%%%%%%%%

\subsection{Proof of Lemma \ref{MC1}}
First, \eqref{a4} obviously implies that for all compact $B$ in $\Theta$, 
$$
\dE \left[ \sup_{t \in B } \left\vert \partial\varphi^{- 1}_t\circ \varphi_{\theta}\left( \varepsilon \right) \right\vert^2\right] <+\infty.
$$
%Moreover, it is well-known (see for instance in \cite{van2000asymptotic} p.305) that if $Y$ is a random variable with distribution function $G$, then for $U \sim \mathcal{U}_{\left[0;1 \right]}$, $Y \sim G^{-1} \left( U \right)$. Hence, 
Moreover, we already saw that the quantile function associated with the distribution of $\varepsilon$ is $F^{-1}$. Consequently, 
\begin{equation}
\label{eq:1}
\dE \left[ \sup_{t \in B } \left\vert \partial\varphi^{- 1}_t\circ \varphi_{\theta}\left( \varepsilon \right) \right\vert^2\right]= \int_0^ 1  \sup_{t \in B } \left\vert \partial\varphi^{- 1}_t\circ \varphi_{\theta}\left( F^{-1}(x) \right) \right\vert^2dx < +\infty.
\end{equation}
Now, it follows from \eqref{a2} that for all $x\in{I_{2}}$,
\begin{equation}
\label{majorationpartial}
\partial \left[ \left(F^{-1}(x)-\varphi_{t}^{-1}\circ\varphi_{\theta}\circ F^{-1}(x)\right)^{2} \right]= -2\partial\varphi_{t}^{-1}\left(\varphi_{\theta}\circ F^{-1}(x)\right)\left(F^{-1}(x)-\varphi_{t}^{-1}\circ\varphi_{\theta}\circ F^{-1}(x)\right)
\end{equation}
is a continuous function with respect to $t$. In addition, if $B$ is a compact set containing $\theta$, it follows from \eqref{a2} together with the mean value Theorem that 
%\begin{eqnarray}
%\nonumber
%F^{-1}(x)-\varphi_{t}^{-1}\circ\varphi_{\theta}\left( F^{-1}(x) \right)&=&\varphi_{\theta}^{-1}\circ\varphi_{\theta}\left(F^{-1}(x)\right)-\varphi_{t}^{-1}\circ\varphi_{\theta}\left( F^{-1}(x) \right)\\
%\label{eq:2}
%&=&\partial\varphi_{\widetilde{t}}^{-1}\circ \varphi_{\theta}\left( F^{-1}(x) \right)\left(\theta  - t\right)
%\end{eqnarray} 
%where $\widetilde{t}$ is between $t$ and $\theta$. Consequently, if $B$ is a compact interval containing $\theta$,  
there exists a constant $C_B>0$ such that
\begin{equation}
\label{Taylorfin}
\sup_{t \in B }  \left\vert  F^{-1}(x)-\varphi_{t}^{-1}\circ \varphi_{\theta}\left( F^{-1}(x) \right)\right\vert \leq  C_B \sup_{t \in B } \left\vert \partial\varphi_{t}^{-1}\circ\varphi_{\theta}\left( F^{-1}(x) \right)\right\vert.
\end{equation}
Hence, we deduce from \eqref{majorationpartial} and the previous inequality that
$$
\sup_{t \in B} \left\vert \partial \left[ \left(F^{-1}(x)-\varphi_{t}^{-1}\circ\varphi_{\theta}\circ F^{-1}(x)\right)^{2} \right]\right\vert \leq   2  C_B\sup_{t \in B }\left\vert \partial\varphi_{t}^{-1}\circ \varphi_{\theta}\left(F^{-1}(x) \right)\right\vert^2
$$
which implies by \eqref{eq:1} that
$$
\sup_{t \in B} \left\vert \partial \left[ \left(F^{-1}(x)-\varphi_{t}^{-1}\circ\varphi_{\theta}\circ F^{-1}(x)\right)^{2} \right]\right\vert
$$
is integrable with respect to $x$.
Finally, $M$ is continuously differentiable on $\Theta$ and for all $t\in{\Theta}$,
$$
M^{\prime}(t)=\int_{0}^{1}-2\partial\varphi_{t}^{-1}\left(\varphi_{\theta}\circ F^{-1}(x)\right)\left(F^{-1}(x)-\varphi_{t}^{-1}\circ\varphi_{\theta}\circ F^{-1}(x)\right)dx.
$$
$\hfill 
\mathbin{\vbox{\hrule\hbox{\vrule height1.5ex \kern.6em
\vrule height1.5ex}\hrule}}$
\subsection{Proof of Lemma \ref{MC2}}
Hypothesis \eqref{a6} implies that 
\begin{equation}
-2\partial\varphi_{t}^{-1}\left(\varphi_{\theta}\circ F^{-1}(x)\right)\left(F^{-1}(x)-\varphi_{t}^{-1}\circ\varphi_{\theta}\circ F^{-1}(x)\right)
\end{equation} 
is continuously differentiable with respect to $t$. In addition, we have 
\begin{eqnarray*}
&&\partial \left[\partial\varphi_{t}^{-1}\left(\varphi_{\theta}\circ F^{-1}(x)\right)\left(F^{-1}(x)-\varphi_{t}^{-1}\circ\varphi_{\theta}\circ F^{-1}(x)\right) \right]\\
&=&-\left[ \partial\varphi_{t}^{-1}\circ\varphi_{\theta}\circ F^{-1}(x)\right]^2+\partial^2\varphi_{t}^{-1}\circ\varphi_{\theta}\circ F^{-1}(x)\left(F^{-1}(x)-\varphi_{t}^{-1}\circ\varphi_{\theta}\circ F^{-1}(x)\right).
\end{eqnarray*}
It follows from \eqref{Taylorfin} that for every compact set $B$ containing $t$ and $\theta$, 
\begin{align*} |\partial^2\varphi_{t}^{-1}\circ\varphi_{\theta}\circ F^{-1}(x) & \left(F^{-1}(x)-\varphi_{t}^{-1}\circ\varphi_{\theta}\circ F^{-1}(x)\right)|\\ & \leqslant  C_B\sup_{t \in B} \left\vert \partial^2\varphi_{t}^{-1}\circ\varphi_{\theta}\circ F^{-1}(x) \right\vert \sup_{t \in B} \left\vert \partial\varphi_{t}^{-1}\circ\varphi_{\theta}\circ F^{-1}(x) \right\vert.
\end{align*}
Then, \eqref{a6} and \eqref{eq:1} together with the Cauchy Schwartz inequality imply that 
$$
\partial^2\varphi_{t}^{-1}\circ\varphi_{\theta}\circ F^{-1}(x)\left(F^{-1}(x)-\varphi_{t}^{-1}\circ\varphi_{\theta}\circ F^{-1}(x)\right)
$$
is integrable with respect to $x$. Hence, we have
$$\int_0^1 \sup_{t \in B}\left\vert \partial \left[\partial\varphi_{t}^{-1}\left(\varphi_{\theta}\circ F^{-1}(x)\right)\left(F^{-1}(x)-\varphi_{t}^{-1}\circ\varphi_{\theta}\circ F^{-1}(x)\right) \right]\right\vert dx <+\infty$$
which enables us to conclude that  $M$ is twice continuously differentiable on $\Theta$ and for all $t\in{\Theta}$,
\begin{align*}
M^{\prime \prime}(t)&=2\int_{0}^{1}\left[ \partial\varphi_{t}^{-1}\circ\varphi_{\theta}\circ F^{-1}(x)\right]^2dx  \\ &- 2\int_{0}^{1}\partial^2\varphi_{t}^{-1}\circ\varphi_{\theta}\circ F^{-1}(x)\left(F^{-1}(x)-\varphi_{t}^{-1}\circ\varphi_{\theta}\circ F^{-1}(x)\right) dx.
\end{align*}
$\hfill 
\mathbin{\vbox{\hrule\hbox{\vrule height1.5ex \kern.6em
\vrule height1.5ex}\hrule}}$
\subsection{Proof of Theorem \ref{ASCV}.}
Denote by $\cF_{n}$ the $\sigma $-algebra of the events occurring up to time $n$,
$\cF_{n}=\sigma(\varepsilon_0,\ldots, \varepsilon_n)$. First of all, we shall calculate the two first conditional
moments of the random variable $T_n$ given by \eqref{DefT}. On the one hand, one has
\begin{eqnarray*}
\dE[T_{n+1}|\mathcal{F}_{n}]&=&-2\dE\Bigl[\partial\varphi_{\wh{\theta}_{n}}^{-1}\left(X_{n+1}\right)\left(\varepsilon_{n+1}-\varphi_{\wh{\theta}_{n}}^{-1}\left(X_{n+1}\right)\right)|\mathcal{F}_{n}\Bigr],\\
&=&-2\dE\Bigl[\partial\varphi_{\wh{\theta}_{n}}^{-1}\circ\varphi_{\theta}(\varepsilon_{n+1})\left(\varepsilon_{n+1}-\varphi_{\wh{\theta}_{n}}^{-1}\circ\varphi_{\theta}(\varepsilon_{n+1})\right)|\mathcal{F}_{n}\Bigr].\\
\end{eqnarray*}
Moreover, as $\varepsilon_{n+1}$ is independent of $\mathcal{F}_{n}$ and $\wh{\theta}_{n}\in{\mathcal{F}_{n}}$, one can deduce from \eqref{defMprime} that
\begin{eqnarray*}
&&-2\dE\Bigl[\partial\varphi_{\wh{\theta}_{n}}^{-1}\circ\varphi_{\theta}(\varepsilon_{n+1})\left(\varepsilon_{n+1}-\varphi_{\wh{\theta}_{n}}^{-1}\circ\varphi_{\theta}(\varepsilon_{n+1})\right)|\mathcal{F}_{n}\Bigr]\\
&=&-2\int_{0}^{1}\partial\varphi_{\wh{\theta}_{n}}^{-1}\circ\varphi_{\theta}\circ F^{-1}(x)\left(F^{-1}(x)-\varphi_{\wh{\theta}_{n}}^{-1}\circ\varphi_{\theta}\circ F^{-1}(x)\right)dx\\
&=&M^{\prime}(\wh{\theta}_{n})
\hspace{1cm}\text{a.s.}
\end{eqnarray*}
which immediately leads to
\begin{equation}
\label{meTn}
\dE[T_{n+1}|\mathcal{F}_{n}]=M^{\prime}(\wh{\theta}_{n})\hspace{1cm}\text{a.s.}
\end{equation}
On the other hand,
\begin{eqnarray}
\nonumber
\dE\Bigl[T_{n+1}^2|\mathcal{F}_{n}\Bigr]&=&4\dE\Bigl[\partial\varphi_{\wh{\theta}_{n}}^{-1}\left(X_{n+1}\right)^{2}\left(\varepsilon_{n+1}-\varphi_{\wh{\theta}_{n}}^{-1}\left(X_{n+1}\right)\right)^{2}|\mathcal{F}_{n}\Bigr],\\
\label{meanT2}
&=&4\dE\Bigl[\partial\varphi_{\wh{\theta}_{n}}^{-1}\left(X_{n+1}\right)^{2}\left(\varphi_{\theta}^{-1}\left(X_{n+1}\right)-\varphi_{\wh{\theta}_{n}}^{-1}\left(X_{n+1}\right)\right)^{2}|\mathcal{F}_{n}\Bigr],
\end{eqnarray}
Moreover, it follows from the mean value Theorem that
\begin{equation}
\label{mvt}
|\varphi_{\theta}^{-1}\left(X_{n+1}\right)-\varphi_{\wh{\theta}_{n}}^{-1}\left(X_{n+1}\right)|\leq{\sup_{t\in{[a;b]}}|\partial\varphi_{t}^{-1}(X_{n+1})|\times|\widehat{\theta}_{n}-\theta|}.
\end{equation}
Consequently, the conjunction of \eqref{meanT2} and \eqref{mvt} leads to
\begin{equation}
\dE\Bigl[T_{n+1}^2|\mathcal{F}_{n}\Bigr]\leq4\left(\widehat{\theta}_{n}-\theta\right)^{2}\dE\Bigl[\sup_{t\in{[a;b]}}|\partial\varphi_{t}^{-1}(X)|^{4}\Bigr].
\end{equation}
Hence,  there exists a positive constant $C_{1}$ given by \eqref{defC1} such that
\begin{equation}
\label{msTn}
\dE\Bigl[T_{n+1}^2|\mathcal{F}_{n}\Bigr]\leq C_{1}\left(\widehat{\theta}_{n}-\theta\right)^{2}\hspace{1cm}\text{a.s.}
\end{equation}
Furthermore, for all $n\geq0$, let $V_{n}=\Big(\wh{\theta}_{n}-\theta\Big)^2$. We clearly have
\begin{eqnarray*}
V_{n+1}&=&\Big(\wh{\theta}_{n+1}-\theta\Big)^2,\\
&=&\Big(\pi_{[a;b]}\Big(\wh{\theta}_{n}-\gamma_{n+1}T_{n+1}\Big)-\theta\Big)^2,\\
&=&\Big(\pi_{[a;b]}\Big(\wh{\theta}_{n}-\gamma_{n+1}T_{n+1}\Big)-\pi_{[a;b]}(\theta)\Big)^2
\end{eqnarray*}
as we have assumed that $\theta$ belongs to $]a;b[$. Since $\pi_{[a;b]}$ is a Lipschitz function with Lipschitz constant $1$, we obtain that
\begin{eqnarray*}
V_{n+1}
&\leq&\left(\wh{\theta}_{n}-\gamma_{n+1}T_{n+1}-\theta\right)^2,\\
&\leq&V_{n}+\gamma_{n+1}^{2}T_{n+1}^{2}-2\gamma_{n+1}T_{n+1}(\wh{\theta}_{n}-\theta).
\end{eqnarray*}
Hence, it follows from \eqref{meTn} together with \eqref{msTn} that
\begin{equation}
\label{RSiegmund}
\dE[V_{n+1}|\mathcal{F}_{n}]\leq V_{n}(1+C_{1}\gamma_{n+1}^{2})-2\gamma_{n+1}(\wh{\theta}_{n}-\theta)M^{\prime}(\wh{\theta}_{n})
\hspace{5mm}\text{a.s.}
\end{equation}
In addition, as $\wh{\theta}_{n} \in [a;b]$, \eqref{a5} implies that $(\wh{\theta}_{n}-\theta)M^{\prime}(\wh{\theta}_{n})>0$.
Then, we deduce from \eqref{RSiegmund} together with Robbins-Siegmund Theorem,
see Duflo \cite{Duflo97} page 18, that the sequence $(V_n)$ converges a.s. to a finite random variable $V$
and
\begin{equation}
\label{sumRS}
\sum_{n=1}^\infty \gamma_{n+1}(\wh{\theta}_{n}-\theta)M^{\prime}(\wh{\theta}_{n})<+\infty \hspace{1cm}\text{a.s.}
\end{equation}
Assume by contradiction that $V \neq 0$ a.s. Then, one can find two constants $c$ and $d$ such that 
$$0<c<d<2\max\left(|a|,|b|\right),$$
and for
$n$ large enough, the event $\{c<| \wh{\theta}_{n}-\theta|<d\}$ is not negligible. However, on this annulus, one can also find some
constant $e>0$ such that $(\wh{\theta}_{n}-\theta)M^{\prime}(\wh{\theta}_{n})\geq e$ which, by \eqref{sumRS}, implies that
$$
\sum_{n=1}^\infty \gamma_{n}<+\infty.
$$
This is of course in contradiction with assumption \eqref{hypgamma}. Consequently, we obtain that $V=0$ a.s.
leading to the almost sure convergence of $\wh{\theta}_{n}$ to $\theta$. 
$\hfill 
\mathbin{\vbox{\hrule\hbox{\vrule height1.5ex \kern.6em
\vrule height1.5ex}\hrule}}$
\subsection{Proof of Theorem \ref{TLC}.}
Our goal is to apply Theorem 2.1 of Kushner and Yin \cite{KushnerYin03} page 330. First of all, as $\gamma_{n}=1/n$, 
the conditions on the decreasing step is satisfied. Moreover, we already saw that $\wh{\theta}_n$ converges almost surely to $\theta$.
Consequently, all the local assumptions of Theorem 2.1 of \cite{KushnerYin03} are satisfied. 
In addition, it follows from $(\ref{meTn})$ that $\dE\left[T_{n+1}|\cF_{n}\right]=M^{\prime}(\wh{\theta}_n)$ a.s. and the function $M$ is 
two times continuously differentiable. 
Hence, $M(\theta)=0$, $M^{\prime}(\theta)=0$ and $M^{\prime\prime}(\theta)>1/2$. 
Furthermore, it follows from \eqref{msTn} and the almost sure convergence of $\wh{\theta}_n$ to $\theta$ that
$$
\lim_{n\rightarrow \infty}\dE\left[T_{n+1}^{2}|\cF_{n}\right]=0\hspace{1cm}\text{a.s.}
$$
Finally, Theorem 4.1 of \cite{KushnerYin03} page 341 ensures that the sequence $(W_n)$ given by
$$
W_n=\sqrt{n}(\wh{\theta}_{n}-\theta)
$$
is tight. Then, one shall deduce from Theorem 2.1 of \cite{KushnerYin03}  that
$$
\sqrt{n}(\wh{\theta}_{n}-\theta) \liml \delta_{0}.
$$
%and
%$$
%\sqrt{n}(\widetilde{\theta}_{n}-\theta) \liml \mathcal{N}\left(0,\Sigma\right)
%$$
%where
%$$
%\Sigma=\sigma^{2}\int_{0}^{+\infty}\exp\left(2(-M^{\prime\prime}(\theta)+1/2)u\right)du=\frac{\sigma^{2}}{2\left(M^{\prime\prime}(\theta)-\frac{1}{2}\right)}.
%$$
Moreover, taking expectation on both sides of \eqref{RSiegmund} leads, for all $n\geq0$, to
\begin{equation}
\label{RSiegmund2}
v_{n+1}\leq v_{n}(1+C_{1}\gamma_{n+1}^{2})-2\gamma_{n+1}\mathbb{E}\left[(\wh{\theta}_{n}-\theta)M^{\prime}(\wh{\theta}_{n})\right]
\end{equation}
where 
$$
v_{n}=\mathbb{E}\left[\left(\wh{\theta}_{n}-\theta\right)^{2}\right].
$$
In addition, as $M^{\prime}(\theta)=0$, one have
\begin{equation}
\label{taylorMprime}
M^{\prime}(\wh{\theta}_{n})=(\wh{\theta}_{n}-\theta)\int_{0}^{1}M^{\prime\prime}(\theta+x(\wh{\theta}_{n}-\theta))dx\hspace{1cm}\text{a.s.}
\end{equation}
Consequently, it follows from \eqref{RSiegmund2} and \eqref{taylorMprime} that
\begin{equation}
\label{RSiegmund3}
v_{n+1}\leq v_{n}(1+C_{1}\gamma_{n+1}^{2})-2\gamma_{n+1}\mathbb{E}\left[(\wh{\theta}_{n}-\theta)^{2}\int_{0}^{1}M^{\prime\prime}(\theta+x(\wh{\theta}_{n}-\theta))dx\right].
\end{equation}
Finally, since $\theta\in{]a;b[}$ and $\wh{\theta}_{n}\in{[a;b]}$, $\theta+x(\wh{\theta}_{n}-\theta)\in{[a;b]}$ for all $x\in{[0;1]}$. Then, as we have supposed that $M^{\prime\prime}(t)\geq{1/2}$ for all $t\in[a;b]$, we can write that
$$
\int_{0}^{1}M^{\prime\prime}(\theta+x(\wh{\theta}_{n}-\theta))dx\geq1/2.
$$
Then, we find from \eqref{RSiegmund3} that for all $n\geq0$,
\begin{equation}
\label{RSiegmund4}
v_{n+1}\leq v_{n}(1+C_{1}\gamma_{n+1}^{2}-\gamma_{n+1}).
\end{equation}
Moreover, the standard convex inequality given for all $x\in{\mathbb{R}}$, by
$$
1-x\leq{\exp(-x)}
$$
implies that
\begin{equation}
\label{RSiegmund5}
v_{n+1}\leq v_{n}\exp\left(C_{1}\gamma_{n+1}^{2}-\gamma_{n+1}\right).
\end{equation}
An immediate recurrence in \eqref{RSiegmund5} leads to
\begin{eqnarray}
\nonumber
v_{n}&\leq& v_{0}\prod_{k=1}^{n}\exp\left(C_{1}\gamma_{k}^{2}-\gamma_{k}\right),\\
\nonumber
&\leq&v_{0}\exp\left(C_{1}\sum_{k=1}^{n}\gamma_{k}^{2}-\sum_{k=1}^{n}\gamma_{k}\right),\\
\label{RSiegmund6}
&\leq&v_{0}\exp\left(C_{1}\sum_{k=1}^{+\infty}\gamma_{k}^{2}-\sum_{k=1}^{n}\gamma_{k}\right).
\end{eqnarray}
As $\gamma_{k}=1/k$, it follows immediately from \eqref{RSiegmund6} together with
$$
\sum_{k=1}^{+\infty}\gamma_{k}^{2}=\frac{\pi^{2}}{6}
$$
and
$$
\sum_{k=1}^{n}\gamma_{k}\geq{\log(n+1)}
$$
that, for all $n\geq0$,
$$
v_{n}\leq v_{0}\frac{\exp\left(C_{1}\pi^{2}/6\right)}{n+1}.
$$
which achieves the proof of Theorem \ref{TLC}.
$\hfill 
\mathbin{\vbox{\hrule\hbox{\vrule height1.5ex \kern.6em
\vrule height1.5ex}\hrule}}$

%%%%%%%%%%%%%%%%%%%%%%%%%%%%%%%%%%%%%%%%%%%%%%%%%%%%%%%%%%%%%%%%%%%%%%%%%%%%%%%%

\section{PROOFS OF THE NONPARAMETRIC RESULTS}

%%%%%%%%%%%%%%%%%%%%%%%%%%%%%%%%%%%%%%%%%%%%%%%%%%%%%%%%%%%%%%%%%%%%%%%%%%%%%%%%
Recall that $f$ is the density of $\varepsilon$ and denote by $f^t$ the density of $Z\left(t \right)$. As the distribution of $Z(t)$ is $F\circ\varphi_{\theta}^{-1}\circ\varphi_{t}$, we have for all $x\in{I_{1}}$,
$$
f^t(x)= f\left(\varphi_{\theta}^{-1} \circ \varphi_t (x) \right) d\left[\varphi_{\theta}^{-1}  \circ \varphi_t \right] (x).
$$
We can note that $f^{\theta}=f$. We start by stating some facts about the densities $f^t(x)$ which will be used hereinafter. 
Firstly, we have
\begin{eqnarray*}
f^t(x)&=& f\left(\varphi_{\theta}^{-1} \circ \varphi_t (x) \right) d\left[\varphi_{\theta}^{-1}  \circ \varphi_t\right] (x)\\
&=&f\left(\varphi_{\theta}^{-1} \circ \varphi_t (x) \right) d \varphi_t(x) d\left[\varphi_{\theta}^{-1}\right]  \left(  \varphi_t (x) \right).
\end{eqnarray*}
Hence, the hypothesis \eqref{ad1}, \eqref{ad2} and \eqref{ad3} implies that $f^t$ is twice continuously differentiable with respect to $x$. Moreover, for all $x\in{I_{1}}$,
$$
df^t(x)= f\left(\varphi_{\theta}^{-1} \circ \varphi_t (x) \right) d^2\left[\varphi_{\theta}^{-1}  \circ \varphi_t \right] (x) +  f '\left(\varphi_{\theta}^{-1} \circ \varphi_t (x) \right) \left( d\left[\varphi_{\theta}^{-1}  \circ \varphi_t\right] (x)\right)^2
$$
and
\begin{align*}
d^2f^t(x)= & f\left(\varphi_{\theta}^{-1} \circ \varphi_t (x) \right) d^3\left[\varphi_{\theta}^{-1}  \circ \varphi_t \right] (x)\\ & + 3 f^{\prime}\left(\varphi_{\theta}^{-1} \circ \varphi_t (x) \right) d\left[\varphi_{\theta}^{-1}  \circ \varphi_t \right] (x) d^2\left[\varphi_{\theta}^{-1}  \circ \varphi_t \right] (x)\\ & +  f^{\prime\prime}\left(\varphi_{\theta}^{-1} \circ \varphi_t (x) \right) \left( d\left[\varphi_{\theta}^{-1}  \circ \varphi_t \right] (x)\right)^3.
\end{align*}
Hence, it follows from \eqref{ad1} to \eqref{ad4} that $f^t(x)$, $df^t(x)$ and $d^2f^t(x)$ are bounded on $\Theta \times I_{1}$. Secondly, \eqref{ad5} implies that $f^t (x)$ is also continuously differentiable with respect to $(t , x)$ and we have for all $t\in{\Theta}$ and for all $x\in{I_{1}}$,
$$
\partial f^t (x) = f\left(\varphi_{\theta}^{-1} \circ \varphi_t (x) \right)\partial  d\left[\varphi_{\theta}^{-1}  \circ \varphi_t \right] (x) +  f '\left(\varphi_{\theta}^{-1} \circ \varphi_t (x) \right)  d\left[\varphi_{\theta}^{-1}  \circ \varphi_t \right] (x)\partial \left[\varphi_{\theta}^{-1}  \circ \varphi_t \right] (x),
$$
where 
$$
\partial \left[\varphi_{\theta}^{-1}  \circ \varphi_t \right] (x)= \partial\varphi_t(x) d \left[\varphi_{\theta}^{-1}   \right] \left( \varphi_t(x)\right),
$$
and 
$$
\partial  d\left[\varphi_{\theta}^{-1}  \circ \varphi_t \right] (x)  = \partial d\varphi_t(x) d \left[\varphi_{\theta}^{-1}   \right] \left( \varphi_t (x)\right)+ \partial\varphi_t(x)d \varphi_{t}(x) d^2 \left[\varphi_{\theta}^{-1}   \right] \left( \varphi_t (x)\right).
$$
Hence, under \eqref{ad4} and \eqref{ad5} 
\begin{equation}
\label{boundft}
\sup_{t\in \Theta} \left\vert \partial f^t (x) \right\vert < +\infty.
\end{equation}
%Finally, recall that if $\beta >-1$, then  $$\lim_{n\rightarrow \infty } \frac{\beta +1 }{n^{1+\beta}}\sum_{i=1}^n i^\beta =1.$$
\subsection{Proof of Theorem \ref{th:cvps}}
Recall that $\mathcal{F}_{n}=\sigma \{ \varepsilon_0,\dots,\varepsilon_{n}\}$ and note that  $\widehat{\theta}_{n-1}$ is measurable with respect to $\mathcal{F}_{n-1}$. Denote, for all $x\in{I_{1}}$,
$$
W_n(x)=\frac {1}{h_n} K \left( \frac{x-{Z}_n(\widehat{\theta}_{n-1})}{h_n} \right).
$$
Then, we have the decomposition for all $x\in{I_{1}}$,
$$
n\widehat{f}_n(x) =M_n(x)+ N_n(x),
$$ where 
\begin{equation}
\label{defMn}
M_n(x) =  \sum_{i=1}^n \E \left[ W_i(x) | \mathcal{F}_{i-1}\right]
\end{equation}
and
\begin{equation}
\label{defNn}
N_n(x) =  \sum_{i=1}^n \left(W_i(x)- \E \left[ W_i(x) | \mathcal{F}_{i-1}\right] \right).
\end{equation}
On the one hand, for a fixed $\widehat{\theta}_{n-1}$, recall that  $f^{\widehat{\theta}_{n-1}}$ denotes the density of  ${Z}_n(\hat{\theta}_{n-1})$. Then, with the changes of variables $v=\frac{x-u}{h_i}$ we have that
\begin{align*}
\E \left[ W_i(x) | \mathcal{F}_{i-1}\right] & = \int_\R \frac {1}{ h_i}K \left( \frac{x-u}{h_i} \right) f^{\widehat{\theta}_{i-1}}(u)du \\ 
& = \int_\R  K(v) f^{\widehat{\theta}_{i-1}} (x-h_iv) dv.
\end{align*} 
Hence,
$$\E \left[ W_i(x) | \mathcal{F}_{i-1}\right] - f^{\widehat{\theta}_{i-1}}(x) =\int_\R \left(f^{\widehat{\theta}_{i-1}}(x-vh_{i})-  f^{\widehat{\theta}_{i-1}} (x)\right) K(v) dv.$$
Moreover, we already saw that $f^{t}$ is twice continuously differentiable. Thus, for all $t\in \Theta$, there exists a real $z_{i}=x-vh_{i}y$, with $0<y<1$, such that
\begin{equation}
\label{Taylorft}
f^t(x-vh_{i})-  f^t (x)= -vh_{i} df^t (x) + \frac{(vh_{i})^2}{2} d^2 f^t (z_{i}).
\end{equation}
Using the parity of $K$ and preliminary remarks on $d^2f^t$, we obtain that 
$$\int_\R \left(f^t(x-vh_{i})-  f^t (x)\right) K(v) dv = \int_\R  \frac{(vh_{i})^2}{2} d^2 f^t (z_{i})K(v) dv$$
which implies that
$$\sup_{t \in \Theta} \left\vert \int_\R \left(f^t(x-vh_{i})-  f^t (x)\right) K(v) dv  \right\vert \leqslant \frac{h_{i}^2}{2} \sup_ {t \in \Theta, z \in I_{1} }\left\vert d^2 f^t (z) \right\vert \int_\R  v^2K(v) dv. $$
Consequently, there exists $C_{2}>0$ such that
\begin{equation}
\label{control1}
\left\vert \E \left[ W_i(x) | \mathcal{F}_{i-1}\right] - f^{\widehat{\theta}_{i-1}}(x) \right\vert \leqslant C_2 h_{i}^ 2.
\end{equation}
Moreover, since $f^{t}$ is a continuous function with respect to $t$,
%%c'est vrai ?
and $\widehat{\theta}_{n}$ converges to $\theta$ almost surely, we have for all $x\in{I_{1}}$,
\begin{equation}
\label{cvpsf}
f^{\widehat{\theta}_{n-1}}(x) \xrightarrow{i \rightarrow \infty} f(x)\hspace{6mm}\textnormal{ a.s.}
\end{equation}
Consequently, Cesaro's Theorem with \eqref{control1} imply that
\begin{equation}
 \label{firstterm}
\frac{1}{n} M_n(x)\xrightarrow{n \rightarrow \infty} f(x) \hspace{6mm}\textnormal{ a.s.}
\end{equation}
On the other hand, since $K$ is bounded, $(N_n(x))$ is a square integrable martingale whose predictable quadratic variation is given by
\begin{eqnarray*}
<N(x)>_n &=& \sum_{i=1}^n \E \left[N_i^2(x) | \mathcal{F}_{i-1}\right] - N_{i-1}^2(x),\\
&=&\sum_{i=1}^n \E \left[ W_i^2(x) | \mathcal{F}_{i-1}\right] - \E^{2} \left[ W_i(x) | \mathcal{F}_{i-1}\right] .
\end{eqnarray*}
Moreover, we also have
$$\E \left[ W_i^2(x) | \mathcal{F}_{i-1}\right] =\frac{1}{h_i} \int  K^2(v) f^{\widehat{\theta}_{i-1}} (x-h_iv) dv.$$
However, \eqref{Taylorft} together with the regularity of $f^t(x)$ and the parity of $K$ imply that
$$\sup_{t \in \Theta} \left\vert \int_\R \frac{1}{h_{i}}\left(f^t(x-vh_{i})- f^t (x)\right) K^2(v) dv  \right\vert \leqslant \frac{h_{i}}{2} \sup_ {t \in \Theta, z \in I_{1} }\left\vert d^2 f^t (z) \right\vert  \int_\R  v^2K^2(v) dv. $$
Consequently, there exists $C_3>0$ such that
\begin{equation}
\label{control2}
\left\vert \E \left[ W^2_i(x) | \mathcal{F}_{i-1}\right] - \frac{\nu^2}{h_i}f^{\widehat{\theta}_{i-1}}(x) \right\vert \leqslant C_3 h_{i}
\end{equation}
where $\nu^2 = \int_\R K^2(u)du$.
It also follows from \eqref{cvpsf} and Toeplitz Lemma that 
$$
\lim_{n\rightarrow \infty } \frac{1}{\sum_{i=1}^n h_i^{-1}}\sum_{i=1}^n\frac{1}{ h_i}f^{\widehat{\theta}_{i-1}}(x) =f(x) \hspace{6mm}\textnormal{ a.s.}
$$
In addition, we deduce from the elementary equivalence
$$
\sum_{i=1}^n \frac{1}{ h_i} \sim \frac{n^{1+\alpha}}{\alpha +1}
$$
that
$$
\lim_{n\rightarrow \infty } \frac{1}{n^{1+\alpha}}\sum_{i=1}^n\frac{\nu^2}{ h_i}f^{\widehat{\theta}_{i-1}}(x) = \frac{\nu^2}{\alpha +1 }f(x) \hspace{6mm}\textnormal{ a.s.}
$$
Finally, \eqref{control2} leads to
\begin{equation}
\label{mean2Wi}
\lim_{n\rightarrow \infty } \frac{1}{n^{1+\alpha}}\sum_{i=1}^n \E \left[ W_i^2(x) | \mathcal{F}_{i-1}\right] =  \frac{\nu^2}{\alpha +1 }f(x) \hspace{6mm}\textnormal{ a.s.}
\end{equation}
Moreover, \eqref{control1} together with \eqref{cvpsf} and Cesaro's Theorem imply that
\begin{equation}
\label{secondpartN}
\lim_{n\rightarrow \infty } \frac{1}{n}\sum_{i=1}^n \E^{2} \left[ W_i(x) | \mathcal{F}_{i-1}\right] = f^{2}(x)
\end{equation}
Then, as $\alpha>0$, we can conclude from \eqref{mean2Wi} and \eqref{secondpartN} that
$$
\lim_{n\rightarrow \infty }\frac{ <N(x)>_n}{n^{1+\alpha}} = \frac{\nu^2}{\alpha +1 }f(x) \hspace{6mm}\textnormal{ a.s.}
$$
Consequently, we obtain from the strong law of large numbers for martingales given e.g. by Theorem 1.3.15 of \cite{Duflo97} that for any $\gamma>0$,  $\left(N_n(x) \right)^2= o\left( n^{1+\alpha} \left(\log(n)\right) ^{1+\gamma}\right)$ a.s. which ensures that for all $x \in I_{1}$, 
\begin{equation} \label{secondterm}
\frac{1}{n} N_n(x)\xrightarrow{n \rightarrow \infty}  0 \quad {\rm a.s.}
\end{equation}
Finally, combining  \eqref{firstterm} and \eqref{secondterm}, one obtain that for all $x \in I_{1}$,
\begin{equation} 
\widehat{f}_n(x) \xrightarrow{n \rightarrow \infty}  f(x) \quad {\rm a.s.}
\end{equation}
ending the proof of Theorem \ref{th:cvps}.
$\hfill 
\mathbin{\vbox{\hrule\hbox{\vrule height1.5ex \kern.6em
\vrule height1.5ex}\hrule}}$

\subsection{Proof of Theorem \ref{th:cvl2}}
 Our aim is now to show that for all $x\in I_{1}$,
 $$\E\left[ \left\vert \widehat{f}_n\left( x\right) - f(x) \right\vert^2 \right] \xrightarrow{n \rightarrow \infty} 0 .$$
It follows from the classical decomposition bias-variance that
\begin{equation}
\label{decompomean}
\E\left[ \left\vert \widehat{f}_n\left( x\right) - f(x) \right\vert^2 \right]=B_{n}(x)+V_{n}(x)
\end{equation}
where
\begin{equation}
\label{defBn}
B_{n}(x)=\left\vert\E\left[  \widehat{f}_n\left( x\right)\right] - f(x) \right\vert^2
\end{equation}
and
\begin{equation}
\label{defVn}
V_{n}(x)=\E\left[ \left\vert \widehat{f}_n\left( x\right) - \E\left[\widehat{f}_n(x)\right] \right\vert^2 \right]. 
\end{equation}
Firstly, we can write 
\begin{eqnarray*}
\E\left[  \widehat{f}_n\left( x\right)\right] - f(x) &=& \frac{1}{n} \sum_{i=1}^n\E \big[W_i(x)- f(x)\big],\\
&=&\frac{1}{n} \sum_{i=1}^n\E\big[ \E\left[W_i(x) |\mathcal{F}_{i-1}\right]- f(x)\big].
\end{eqnarray*}
In addition, \eqref{control1} implies that
\begin{equation}
\label{cvWi1}
\E \left[\left\vert \E \left[ W_n(x) | \mathcal{F}_{n-1}\right] - f^{\widehat{\theta}_{n-1}}(x) \right\vert \right] \xrightarrow{n \rightarrow \infty} 0
\end{equation}
It also follows from the boundeness of $f^{\widehat{\theta}_{n-1}}(x)$ and \eqref{cvpsf} together with the dominated convergence Theorem that
\begin{equation}
\label{cvWi2}
\E \left[\left\vert f^{\widehat{\theta}_{n-1}}(x)- f(x) \right\vert \right] \xrightarrow{n \rightarrow \infty} 0.
\end{equation}
Hence, we deduce from \eqref{cvWi1} and \eqref{cvWi2} that
$$ 
\E\big[ \E\left[W_i(x)|\mathcal{F}_{n-1}\right]- f(x)\big] \xrightarrow{n \rightarrow \infty} 0,
$$
which implies by Cesaro's Theorem that
$$
\left\vert\E\left[  \widehat{f}_n\left( x\right)\right] - f(x) \right\vert \xrightarrow{n \rightarrow \infty} 0  
$$
leading to
\begin{equation}
\label{cvBn}
B_{n}(x)\xrightarrow{n \rightarrow \infty} 0.
\end{equation}
Secondly, we focus on the variance term $V_n(x)$. For all $1\leq{i}\leq{n}$ and for all $x\in{I_{1}}$, denote by $U_{i}(x)$ the sequence
\begin{equation}
\label{defUix}
U_i(x)=  W_i(x) - \E\left[  W_i(x) \right].
\end{equation}
Then, we have the decomposition
\begin{equation}
\label{decompoV}
V_{n}(x)=\frac{1}{n^2} \sum_{i=1}^n\E\left[   U_i(x) ^2 \right]  + \frac{2}{n^2} \sum_{i=1, i<j}^n\E\left[  U_i(x) U_j(x) \right]. 
\end{equation}
If $i<j$, we have  
$$
\E\left[  U_i(x) U_j(x) | \mathcal{F}_{j-1} \right] = U_i(x) \E\left[   U_j(x) | \mathcal{F}_{j-1} \right].
$$
In addition, \eqref{control1} implies that
%\begin{eqnarray*}
%\E\left[   U_j(x)| \mathcal{F}_{j-1} \right] &=&\E \left[ W_j(x) | \mathcal{F}_{j-1}\right]-\E \left[ W_j(x) \right],\\
%& =&f^{\widehat{\theta}_{j-1}}(x) - \E\left[ f^{\widehat{\theta}_{j-1}}(x)\right] + C_j h_{j }^2.
%\end{eqnarray*}
\begin{eqnarray*}
\left\vert \E\left[   U_j(x)| \mathcal{F}_{j-1} \right]-f^{\widehat{\theta}_{j-1}}(x)+\E\left[ f^{\widehat{\theta}_{j-1}}(x)\right]\right\vert\leq 2C_2 h_{j }^2.
\end{eqnarray*}
%where $\left( C_j\right)_{j\geq 1}$ is a non random and bounded sequence.
Hence, we obtain that
%$$
%\E\left[  U_i(x) U_j(x) | \mathcal{F}_{j-1} \right] = U_i(x) \left( f^{\widehat{\theta}_{j-1}}(x) -f(x) \right) + U_i(x) \left( f(x)-  \E\left[ f^{\widehat{\theta}_{j-1}}(x)\right] \right) + U_i(x)C_j h_{j }^2 .
%$$
$$
-2C_2 h_{j }^2\left\vert U_i(x)\right\vert \leq\E\left[   U_i(x)U_j(x)| \mathcal{F}_{j-1} \right]-U_i(x)f^{\widehat{\theta}_{j-1}}(x)+U_i(x)\E\left[ f^{\widehat{\theta}_{j-1}}(x)\right]\leq 2C_2 h_{j }^2\left\vert U_i(x)\right\vert.
$$
%From now on, $C$ denotes every constant independent of $n$.
Thus, taking expectation in the previous inequality leads to
$$
-2C_2 h_{j }^2\E\left[\left\vert U_i(x)\right\vert\right]\leq\E\left[ U_i(x)U_j(x)\right]-\E\left[U_i(x)f^{\widehat{\theta}_{j-1}}(x)\right]+\E\left[U_i(x)\right]\E\left[ f^{\widehat{\theta}_{j-1}}(x)\right]\leq 2C_2 h_{j }^2\E\left[\left\vert U_i(x)\right\vert\right].
$$
Finally, we obtain that
\begin{equation}
\label{CS1part}
\left\vert\E\left[ U_i(x)U_j(x)\right]\right\vert\leq{\left\vert\E\left[U_i(x)f^{\widehat{\theta}_{j-1}}(x)\right]-\E\left[U_i(x)\right]\E\left[ f^{\widehat{\theta}_{j-1}}(x)\right]\right\vert+ 2C_2 h_{j }^2\E\left[\left\vert U_i(x)\right\vert\right]}.
\end{equation}
Moreover, we have the following equality
\begin{equation}
\label{CS2part}
\E\left[U_i(x)f^{\widehat{\theta}_{j-1}}(x)\right]-\E\left[U_i(x)\right]\E\left[ f^{\widehat{\theta}_{j-1}}(x)\right]=\E\left[U_i(x)\left(f^{\widehat{\theta}_{j-1}}(x)-f(x)\right)\right]+\left(f(x)-\E\left[ f^{\widehat{\theta}_{j-1}}(x)\right]\right)\E\left[U_i(x)\right].
\end{equation}
%\begin{align*}
%\E\left[  U_i(x) U_j(x) \right] &=\E\left[ U_i(x) \left( f^{\widehat{\theta}_{j-1}}(x) -f(x) \right) \right] + \left( f(x)-  \E\left[ f^{\widehat{\theta}_{j-1}}(x)\right] \right)\E\left[ U_i(x) \right]  +  C_j h_{j }^2\E\left[ U_i(x)\right],\\
%&=\E\left[ U_i(x) \left( f^{\widehat{\theta}_{j-1}}(x) -f(x) \right) \right] + \E\left[  f(x)-  f^{\widehat{\theta}_{j-1}}(x) \right]\E\left[ U_i(x) \right]  +  C_j h_{j }^2\E\left[ U_i(x)\right].
%\end{align*}
Consequently, \eqref{CS1part} and \eqref{CS2part} together with Cauchy-Schwartz's inequality imply that
\begin{equation}
\label{CSUiUj}
\E\left[ | U_i(x) U_j(x)| \right]\leq2\sqrt{\E\left[ U_i(x)^{2} \right]} \left(\sqrt{\E\left[\left( f^{\widehat{\theta}_{j-1}}(x) -f(x) \right)^{2} \right] }+C_2 h_{j }^2\right).
\end{equation}
The definition \eqref{defUix} of $U_{i}(x)$ also leads to
%$$
%\E\left[|U_i(x)|\right]\leq{2\E\left[|W_i(x)|\right]}
%$$
%leads by \eqref{control1} to
%\begin{equation}
%\label{controlUi}
%\E\left[|U_i(x)|\right]\leq{2\left(\E\left[f^{\wh{\theta}_{i-1}}(x)\right]+Ch_{i}^{2}\right)}
%\end{equation}
%and
$$
\E\left[U^{2}_i(x)\right]  \leq{\E\left[W^{2}_i(x)\right]}
$$
which implies by \eqref{control2} that
\begin{equation}
\label{controlUi2}
\E\left[U_i^{2}(x)\right]\leq{\frac{\nu^{2}}{h_{i}}\E\left[f^{\widehat{\theta}_{i-1}}(x)\right]+C_3h_{i}}.
\end{equation}
From now, denote by $C$ a constant which does not depend on $n$.
On the one hand, recall that \eqref{mean2theta} implies that for all $n\geq0$,
\begin{equation}
\label{quadraticmeantheta}
\E\left[ \left\vert {\widehat{\theta}_{n}} -\theta \right\vert^ 2 \right] \leq \frac{C}{{n}}.
\end{equation} 
On the other hand, using the regularity of $f$, we obtain that for all $x\in{I_{1}}$,
 $$
 \left\vert {f}^{t}\left( x\right) - {f}\left( x\right)  \right\vert \leq \sup_{t \in \Theta} \left\vert \partial f^t(x) \right\vert  \left\vert t-\theta\right\vert.
 $$ 
Hence, \eqref{boundft} and \eqref{quadraticmeantheta} lead to
\begin{equation}
\label{mean2f}
\sqrt{\E\left[ \left\vert f^{\widehat{\theta}_{n-1}}(x) -f(x)\right\vert^{2} \right] }\leq \frac{C}{\sqrt{n}}.
\end{equation}
Then, the conjunction of \eqref{CSUiUj}, \eqref{controlUi2} and \eqref{mean2f} implies that
\begin{equation}
\label{CSUiUj2}
\E\left[ | U_i(x) U_j(x)| \right]\leq 2\left( \sqrt{\frac{\nu^{2}}{h_{i}}\E\left[f^{\widehat{\theta}_{i-1}}(x)\right]+C_3h_{i}}\right)\left( \frac{C}{\sqrt{j}} +C_2 h_{j }^2 \right).
\end{equation}
Finally, using the boundedness of $f^t(x)$, we obtain that
\begin{equation}
\label{majoUiUj}
\E\left[ | U_i(x) U_j(x)| \right]\leqslant C\left(\frac{1}{\sqrt{j h_{i}}}+\frac{h_{j}^{2}}{\sqrt{h_i}}\right).
\end{equation}
Moreover, if $h_{n}=1/n^{\alpha}$, one have
\begin{equation*}
 \sum_{i=1, i<j}^n\frac{1}{\sqrt{j h_{i}}}=\sum_{j=2}^n\frac{1}{j^{1/2}} \sum_{i=1}^{j-1}{i}^{\alpha/2}\leq\sum_{j=2}^n\frac{{j}^{\alpha/2 +1}}{j^{1/2}}\leq n^{\frac{3+\alpha}{2}}
\end{equation*}
and
\begin{equation*}
\sum_{i=1, i<j}^n\frac{h_{j}^{2}}{\sqrt{h_i}}=\sum_{j=2}^n h_{j}^{2}\sum_{i=1}^{j-1}i^\frac{\alpha}{ 2}\leq\sum_{j=2}^n\frac{j^{\frac{\alpha}{ 2} +1}}{{j}^{2\alpha}}\leq n^{2- 3 \frac{\alpha}{4}}.
\end{equation*}
%with $1- \frac{3\alpha}{ 2}>-1$ since $\alpha<1$. 
Consequently,
one deduce from the two elementary previous calculations and from \eqref{majoUiUj} that
\begin{equation}
\label{eqfinalUiUj}
\frac{1}{n^{2}} \sum_{i=1, i<j}^n\E\left[ | U_i(x) U_j(x)| \right] \leq C \left(n^{\frac{-1+\alpha}{2}}+n^{-3 \frac{\alpha}{4}}\right)
\end{equation}
which tends to $0$ as $n$ goes to infinity, as $0<\alpha<1$.
In addition, thanks to \eqref{controlUi2} and the boundeness of $f^{t}$, we have
\begin{equation}
\label{eqfinalUi2}
\frac{1}{n^{2}}\sum_{i=1}^n\E\left[   U^2_i(x)  \right]\leq C \frac{1}{n^{2}}\sum_{i=1}^n \frac{1}{h_i} \leq C \frac{n^{\alpha +1}}{n^2} \leq  Cn^{-1+\alpha}
\end{equation}
which tends to $0$ as $n$ goes to infinity, as $\alpha<1$.
Hence, \eqref{decompoV}, \eqref{eqfinalUiUj} together with \eqref{eqfinalUi2} let us to conclude that for all $x\in{I_{1}}$,
\begin{equation}
\label{cvVn}
V_{n}(x)\xrightarrow{n \rightarrow \infty} 0.
\end{equation}
Finally, \eqref{decompomean}, \eqref{cvBn} and \eqref{cvVn} let us to achieve the proof of Theorem \ref{th:cvl2}.
$\hfill 
\mathbin{\vbox{\hrule\hbox{\vrule height1.5ex \kern.6em
\vrule height1.5ex}\hrule}}$

\bibliographystyle{abbrv}
\bibliography{deformRV} 

\end{document}